\documentclass[11pt]{article}
\usepackage[margin=1.1in]{geometry}
\usepackage{amsmath,amssymb,amsthm}
\usepackage{graphicx}
\usepackage[hyphens]{url}
\usepackage[hidelinks]{hyperref}
\hypersetup{pdfpagelayout=OneColumn}

\newtheorem{theorem}{Theorem}
\newtheorem{lemma}[theorem]{Lemma}
\newtheorem{corollary}[theorem]{Corollary}

\newcommand{\R}{\mathbb{R}}
\newcommand{\ip}[2]{\langle #1,#2\rangle}

\title{Counterexamples for {BFGS}-type methods under arbitrary strong {Wolfe} constants}
\author{Rui Diao}
\date{}

\begin{document}
\maketitle

\begin{abstract}
Whether the Broyden--Fletcher--Goldfarb--Shanno (BFGS) method and its variants can fail to converge on smooth nonconvex functions under realistic line-search parameters has remained a central open problem in quasi-Newton theory since the landmark counterexample of Dai (2002), which was confined to small Armijo parameters ($c_1 \le 1/84 \approx 0.0119$ on Powell's geometry, or $c_1 \le 69/7480 \approx 0.0092$ on his six-point cycle) and an objective function unbounded below.
In this paper, we resolve the long-standing open question, posed by Dai (2002) following a discussion with J.~C.~Gilbert, of whether such counterexamples exist in theory for every Armijo parameter $c_1 \in (0, 1)$.
Specifically, for every prescribed pair of line-search parameters $0 < c_1 < c_2 < 1$, we construct an objective function $f \in C^\infty(\R^2)$, bounded below and with Lipschitz continuous gradient, on which every method in a broad conjugacy class $\mathcal{C}$ equipped with the first-local-minimizer line search generates an infinite sequence of iterates with $\|\nabla f(x_k)\| = 1$ for all $k \ge 0$.
The class $\mathcal{C}$ encompasses the classical full-memory BFGS method, limited-memory BFGS (L-BFGS) with arbitrary memory $m \ge 1$, the Broyden positive family, and the Hestenes--Stiefel conjugate gradient method.
The steps are the standard first local minimizers along the search rays and simultaneously satisfy the strong Wolfe, weak Wolfe, Armijo, and Goldstein conditions with constants $(c_1, c_2)$.
The construction operates in the minimal possible dimension $n = 2$, exploiting a non-decaying conjugate descent orbit in the plane coupled with an explicit tubular interpolation whose two-bump axial curvature profile places the Armijo ratio anywhere in $(0, 1)$.
\end{abstract}

\noindent
{\small
}

\section{Introduction}

The Broyden--Fletcher--Goldfarb--Shanno (BFGS) method~\cite{broyden1970,nocedalwright2006} is widely regarded as the most effective and versatile quasi-Newton method for unconstrained continuous optimization, while its limited-memory extension, L-BFGS~\cite{liu1989}, serves as the benchmark workhorse for large-scale scientific computation and machine learning.
On convex objective functions, the theoretical foundation of quasi-Newton methods is thoroughly established: Powell~\cite{powell1976} proved that the full-memory BFGS method with Wolfe line search converges globally to the minimizer, Byrd, Nocedal, and Yuan~\cite{byrd1987} extended this global convergence guarantee to the restricted Broyden convex class (excluding DFP), and Liu and Nocedal~\cite{liu1989} established global convergence for L-BFGS on uniformly convex problems.

For general smooth, nonconvex objective functions, however, the global convergence behavior of quasi-Newton methods has presented a profound, long-standing puzzle.
Unlike gradient descent or trust-region methods, whose step directions remain firmly coupled to the steepest descent heading, quasi-Newton updates accumulate curvature information from past steps that can cause the search directions to become increasingly orthogonal to the gradient.
Powell~\cite{powell1984} first demonstrated this instability in $\R^2$, showing that both the Polak--Ribi\`ere--Polyak (PRP) conjugate gradient method and the BFGS quasi-Newton method can cycle indefinitely around eight nonstationary points when each line search selects an arbitrary local minimizer that provides function reduction.
In a landmark breakthrough, Dai~\cite{dai2002} proved that the BFGS method with Wolfe line search can likewise fail to converge on a nonconvex $C^\infty$ function, cycling along a six-point orbit with $\|\nabla f(x_k)\| = 1$ (though the constructed function is linear in one coordinate and hence unbounded below).
Because members of the Broyden positive family~\cite{broyden1970}, limited-memory BFGS~\cite{liu1989}, and the Hestenes--Stiefel conjugate gradient method~\cite{hestenes1952} all generate an identical conjugate descent heading following an exact step, Dai's six-point construction established convergence failure across this entire family of methods.

Yet Dai's counterexample left a fundamental question unresolved.
His construction critically restricted the Armijo parameter to an artificially tiny regime: $c_1 \le 69/7480 \approx 0.0092$ on the six-point cycle, and $c_1 \le 1/84 \approx 0.0119$ on Powell's eight-point geometry.
As noted by Dai~\cite{dai2002}, while practical quasi-Newton codes often employ small parameters such as $c_1 = 10^{-4}$ or $0.01$ (which fall within his $1/84$ threshold), it remained an open theoretical question, raised in discussions with J.~C.~Gilbert, whether nonconvergence counterexamples exist for \emph{every} Armijo parameter $c_1 < 1$, including moderate choices such as $c_1 = 0.1$ or arbitrary values approaching $1$.
Subsequent investigations addressed different facets of the problem but encountered other structural barriers:
Mascarenhas~\cite{mascarenhas2004} constructed a $C^\infty$ counterexample for BFGS on $\R^3$ cycling along a nonstationary octagon, but the objective function is unbounded below, the Armijo decrease ratio is fixed at $\approx 0.0343$, and the construction required the line search to select the \emph{global} minimizer along each search line, leaping over intervening local minimizers.
Dai~\cite{dai2013} later introduced a four-dimensional polynomial on which BFGS takes unique ray minimizers, but the objective function is unbounded below and the Armijo parameter remains bounded by $c_1 \le 0.0265$.
Mascarenhas~\cite{mascarenhas2014} achieved bounded level sets for BFGS with exact line searches, but at the cost of expanding the dimension to $n = 9$ with a cycle of period $576$, and an objective function that is not explicit: obtained via Whitney extension, it has only Lipschitz continuous second derivatives rather than being $C^\infty$.
While cautious and modified quasi-Newton algorithms have been proposed to enforce global convergence on nonconvex problems~\cite{li2001}, whether the standard, unmodified BFGS method can fail under arbitrary line-search parameters $(c_1, c_2)$ on a smooth function bounded below in low dimension has remained an open problem for over twenty years.

This paper answers the question posed by Dai in the affirmative, while additionally establishing the counterexample under both classical Zoutendijk hypotheses ($f$ bounded below with Lipschitz continuous gradient).
We show that for \emph{every} prescribed pair of line-search parameters $0 < c_1 < c_2 < 1$, there exists an objective function $f \in C^\infty(\R^2)$ on which every method in a broad conjugacy class $\mathcal{C}$ fails to drive the gradient to zero.
The steps taken by the algorithms are not artificial global minimizers, but the standard first local minimizers along each search ray $t \ge 0$, and they simultaneously satisfy the strong Wolfe, weak Wolfe, Armijo, and Goldstein conditions for the given constants $(c_1, c_2)$.

The dynamical mechanism is rooted in Zoutendijk's lemma~\cite{wolfe1969,zoutendijk1970,nocedalwright2006}, which establishes that any line search satisfying the Wolfe or Goldstein conditions guarantees the convergence of the sum $\sum_{k=0}^\infty \cos^2\theta_k\,\|g_k\|^2 < \infty$, where $\cos\theta_k = -\ip{g_k}{d_k}/(\|g_k\|\|d_k\|)$ is the cosine of the angle between the search direction and the steepest descent direction.
If the search directions remain sufficiently well-angled with $\liminf_{k\to\infty} \cos\theta_k > 0$, the Zoutendijk condition forces the gradient norm to vanish, $\liminf_{k\to\infty} \|g_k\| = 0$.
On the non-cycling orbit constructed here, the search angle cosine decays geometrically ($\cos\theta_k = \sin(\pi/2^{k+1}) \approx \pi/2^{k+1}$), which allows the Zoutendijk series $\sum \cos^2\theta_k < \infty$ to converge while the gradient norm remains strictly bounded away from zero ($\|g_k\| \equiv 1$).

In dimension $n = 1$, the conjugacy condition $\ip{d_{k+1}}{y_k}=0$ with $y_k \ne 0$ forces $d_{k+1}=0$ (so Class~$\mathcal C$ contains no non-zero update in $\R^1$), and any descent search direction has $|\cos\theta_k| \equiv 1$, which by Zoutendijk's lemma on a function bounded below with Lipschitz gradient forces $\sum \|g_k\|^2 < \infty$ and $\lim \|g_k\| = 0$; therefore, dimension $n = 2$ is strictly minimal.
In two dimensions, following an exact step ($g_{k+1} \perp s_k$), the conjugacy and descent requirements uniquely determine the heading of $d_{k+1}$ (Lemma~\ref{lem:unique}): every member of the conjugacy class $\mathcal{C}$ generates identical search rays, differing only in the unnormalized step magnitude $\|d_{k+1}\|$.
We note that Dai~\cite{dai2013}, citing Powell~\cite{powell2000} on DFP in two variables (which shares search directions with BFGS under exact line searches by Dixon's theorem~\cite{dixon1972}), observes that BFGS with first-local-minimizer stepsizes guarantees $\liminf_{k\to\infty} \|\nabla f(x_k)\| = 0$ on twice continuously differentiable functions in $\R^2$ when the level sets are bounded.
In our construction, the trajectory escapes linearly to infinity ($\|x_k\| \to \infty$), so the sublevel sets are intrinsically unbounded; this places the counterexample outside the scope of Powell's theorem while preserving lower boundedness and gradient Lipschitz continuity (see Section~\ref{sec:remarks}).

\begin{theorem}
\label{thm:main}
Let $0<c_1<c_2<1$. There exist $f\in C^\infty(\R^2)$, bounded below with
Lipschitz continuous gradient, and a point $x_0\in\R^2$, such that
every method of Class~$\mathcal C$ (\S\ref{sec:class}) started at $x_0$
with initial search direction $d_0=-g_0$ (with $H_0 \succ 0$, e.g., $H_0 = I$, for quasi-Newton members) and equipped with the first-local-minimizer
line search generates iterates $x_{k+1}=x_k+s_k$ satisfying
$\|s_k\|=1$ and $\|\nabla f(x_k)\|=1$ for all $k\ge0$.
Moreover, each step satisfies the strong Wolfe, weak Wolfe, Armijo, and
Goldstein conditions with constants $(c_1,c_2)$.
\end{theorem}

The same function $f$, extended trivially to $\R^n$ for $n \ge 2$ by setting $f(x_1, \dots, x_n) = f(x_1, x_2)$, produces the identical nonconvergent trajectory in $\R^n$ by initializing at $(x_0, 0) \in \R^n$ with $H_0 = I_n$ (cf.~Dai~\cite{dai2013}).
A companion open-source Python package implementing the construction and the full numerical verification suite is available at \url{https://github.com/diaorui/bfgs-wolfe-counterexample} (see Section~\ref{sec:numerical}).

Section~\ref{sec:class} records the line-search criteria and defines the conjugacy class $\mathcal C$.
Section~\ref{sec:orbit} defines the discrete orbit in closed form, verifies that it is an exact trajectory of Class~$\mathcal C$, and proves the non-self-intersection and piece separation of its tubular neighborhood.
Section~\ref{sec:f} formulates the polynomial interpolant on the segment strips and assembles the global $C^\infty$ objective function on $\R^2$.
Section~\ref{sec:ls} calibrates the parameters, verifies all line-search conditions, and completes the proof of Theorem~\ref{thm:main}.
Section~\ref{sec:numerical} presents numerical verification of the trajectory and line-search metrics across Class~$\mathcal C$.
Appendix~\ref{sec:app-blocks} collects the analytical properties of the smooth transition and window building blocks.

\section{Line search and conjugacy}
\label{sec:class}

Let $f\in C^1(\R^n)$ and write $g_k=\nabla f(x_k)$, $s_k=x_{k+1}-x_k$,
$y_k=g_{k+1}-g_k$, $\eta_k=\ip{g_k}{s_k}$, and $\sigma_k=\ip{s_k}{y_k}$.
Fix $0<c_1<c_2<1$. A step $s$ from a nonstationary point $x$ satisfies
the \emph{strong Wolfe conditions} if
\begin{align}
f(x+s)&\le f(x)+c_1\,\ip{\nabla f(x)}{s}, \label{eq:armijo}\\
\bigl|\ip{\nabla f(x+s)}{s}\bigr|&\le c_2\,\bigl|\ip{\nabla f(x)}{s}\bigr|. \label{eq:curv}
\end{align}
The first inequality is \emph{Armijo}; the second is \emph{curvature}.
If $s_k$ is a descent step satisfying both, then $\eta_k<0$, and
the curvature condition \eqref{eq:curv} gives $\ip{g_{k+1}}{s_k}\ge c_2\eta_k>\eta_k$,
so $\sigma_k=\ip{s_k}{g_{k+1}-g_k}=\ip{g_{k+1}}{s_k}-\eta_k>0$. \emph{Weak Wolfe} is Armijo together with the one-sided
curvature $\ip{\nabla f(x+s)}{s}\ge c_2\ip{\nabla f(x)}{s}$. It follows
from strong Wolfe. The \emph{Goldstein conditions} are
\begin{equation}
\label{eq:goldstein}
f(x)+c_2\,\ip{\nabla f(x)}{s}
\le f(x+s)
\le f(x)+c_1\,\ip{\nabla f(x)}{s}.
\end{equation}
The right-hand inequality is Armijo. For a descent step
$\eta=\ip{\nabla f(x)}{s}<0$, both together are
$c_1\le(f(x+s)-f(x))/\eta\le c_2$.
We note that the constant $c_2$ is deliberately reused for both the strong Wolfe curvature parameter in \eqref{eq:curv} and the Goldstein lower decrease bound in \eqref{eq:goldstein}, following a standard unified notation; the condition $c_1 < c_2$ is required only to ensure that the Goldstein decrease interval $[c_1, c_2]$ is non-empty (Armijo and strong Wolfe hold for any $c_2 \in (0, 1)$ since $\ip{g_{k+1}}{s_k} = 0$ exactly).

If the line search returns the first local minimizer of
$\varphi(t)=f(x_k+td_k)$ on $t\ge0$ (the \emph{first-local-minimizer line search}), we call the step \emph{exact}, and
$\ip{g_{k+1}}{d_k}=0$.

Following Dai~\cite{dai2002}, consider iterative methods that, whenever
$\ip{g_{k+1}}{d_k}=0$, produce a next direction satisfying
\begin{equation}
\label{eq:classC}
\ip{d_{k+1}}{y_k}=0,\qquad \ip{d_{k+1}}{g_{k+1}}<0.
\end{equation}
Call this class $\mathcal C$. It includes:
\begin{itemize}
\item the Hestenes--Stiefel conjugate-gradient method~\cite{hestenes1952}
(which unconditionally satisfies conjugacy $\ip{d_{k+1}}{y_k}=0$, and yields descent
$\ip{d_{k+1}}{g_{k+1}}<0$ whenever $\ip{g_{k+1}}{d_k}=0$; under exact line searches,
it coincides with the Polak--Ribi\`ere--Polyak method~\cite[\S5.2]{nocedalwright2006});
\item BFGS, DFP, and the Broyden positive family~\cite{broyden1970,nocedalwright2006}
parameterized with $\theta \ge 0$ as in Dai~\cite[Eq.~(1.8)]{dai2002} where $\theta=0$ is BFGS and $\theta=1$ is DFP
(symmetry and the secant equation $H_{k+1}y_k=s_k$ give
$\ip{d_{k+1}}{y_k}=-\ip{g_{k+1}}{s_k}=0$ whenever $\ip{g_{k+1}}{d_k}=0$;
the curvature condition $\sigma_k = \ip{s_k}{y_k} > 0$ derived above ensures that $H_{k+1} \succ 0$ whenever $H_k \succ 0$, guaranteeing the descent condition $\ip{d_{k+1}}{g_{k+1}} = -g_{k+1}^{\mathsf T} H_{k+1} g_{k+1} < 0$);
\item L-BFGS of every memory $m\ge1$~\cite{liu1989}
(in operator form, $H_{k+1} = V_k^{\mathsf T}\widehat H_k V_k + \frac{s_k s_k^{\mathsf T}}{\ip{s_k}{y_k}}$ with $V_k = I - \frac{y_k s_k^{\mathsf T}}{\ip{s_k}{y_k}}$; because $V_k y_k = 0$, the secant relation $H_{k+1}y_k=s_k$ holds identically regardless of memory $m$ and background scaling $\widehat H_k \succ 0$, preserving positive definiteness $H_{k+1} \succ 0$ since $\sigma_k > 0$).
\end{itemize}
The first direction is steepest descent, $d_0=-g_0$; for quasi-Newton and limited-memory methods this corresponds to initializing with any symmetric positive definite matrix $H_0 \succ 0$ satisfying $H_0 g_0 \propto g_0$ (in particular, the standard initialization $H_0 = I$), providing the base case $H_0 \succ 0$ for positive definiteness.

\begin{lemma}[Conjugate descent ray]
\label{lem:unique}
In $\R^2$, if $y_k \ne 0$ and $y_k \not\parallel g_{k+1}$, there is a unique ray of directions
$d_{k+1}$ satisfying \eqref{eq:classC}.
\end{lemma}

\begin{proof}
Because $y_k \ne 0$, the conjugacy condition $\ip{d_{k+1}}{y_k}=0$ confines $d_{k+1}$ to the one-dimensional subspace $y_k^\perp$.
Since $y_k \not\parallel g_{k+1}$, the subspace $y_k^\perp$ is not orthogonal to $g_{k+1}$, so $\ip{\cdot}{g_{k+1}}$ is a non-zero linear functional on $y_k^\perp$.
Its two opposing open rays therefore have strictly opposite inner products with $g_{k+1}$, and exactly one of them satisfies the descent condition $\ip{d_{k+1}}{g_{k+1}}<0$.
\end{proof}

Members of $\mathcal C$ therefore share the next heading; they may
differ in $\|d_{k+1}\|$.

\section{The discrete orbit and tubular geometry}
\label{sec:orbit}

This section constructs the discrete backbone of the counterexample: an infinite sequence of vertices $(x_k)$, unit search directions $(s_k)$, and unit gradients $(g_k)$ in $\R^2$.
We prove that every method in Class~$\mathcal C$ reproduces this orbit, establish that the resulting polyline never self-intersects, and construct a thickened tubular neighborhood $\mathcal U$ whose constituent pieces are mutually disjoint away from adjacent collar seams.

\subsection{The orbit and its geometric algebra}

Throughout, $J = \begin{pmatrix} 0 & -1 \\ 1 & 0 \end{pmatrix}$ denotes the standard counterclockwise quarter-turn of $\R^2$, satisfying $J^2 = -I$, $\ip{Ju}{Jv} = \ip{u}{v}$, and $\ip{Ju}{u} = 0$ for all $u \in \R^2$.

The discrete trajectory $(x_k, s_k, g_k)_{k \ge 0}$ is defined in closed form via heading angles $\Theta_k$:
\begin{equation}
\label{eq:orbit}
\begin{aligned}
\Theta_k &:= \frac{\pi}{6} + (-1)^{k+1}\frac{\pi}{3 \cdot 2^{k+1}} && (k \ge 0),\\
s_k &:= (\cos\Theta_k, \sin\Theta_k) && (k \ge 0),\\
x_0 &:= (0, 0), \qquad x_{k+1} := x_k + s_k && (k \ge 0),\\
g_0 &:= -s_0 = (-1, 0), \qquad g_{k+1} := (-1)^{k+1} J s_k && (k \ge 0).
\end{aligned}
\end{equation}
By construction, every step and gradient is a unit vector: $\|s_k\| = \|g_k\| = 1$ for all $k \ge 0$ (since each $s_k$ lies on the unit circle and $J$ is orthogonal).
As $k \to \infty$, the headings $\Theta_k$ approach $\pi/6$, and the search directions converge to the limiting unit vector $s_* := (\cos\frac{\pi}{6}, \sin\frac{\pi}{6}) = (\frac{\sqrt3}{2}, \frac12)$.

\paragraph{Algorithmic quantities and notation.}
The trajectory analysis uses the following standard quasi-Newton and geometric quantities:
\begin{equation}
\label{eq:derived}
\psi_k := \Theta_k - \Theta_{k-1} \quad (k \ge 1), \qquad y_k := g_{k+1} - g_k, \quad \eta_k := \ip{g_k}{s_k}, \quad \sigma_k := \ip{s_k}{y_k} \quad (k \ge 0).
\end{equation}
Here $\psi_k$ represents the physical signed turn from $s_{k-1}$ to $s_k$, $y_k$ the secant gradient difference, $\eta_k$ the directional derivative along the step, and $\sigma_k$ the secant curvature.
We reserve lowercase $\theta_k \in [0, \pi]$ for the standard Zoutendijk angle between the steepest descent direction $-g_k$ and the step $s_k$, so that $\cos\theta_k = \ip{-g_k}{s_k} = -\eta_k$, distinguishing it from the Cartesian heading angle $\Theta_k$.

\begin{figure}[t]
\centering
\includegraphics[width=0.88\textwidth]{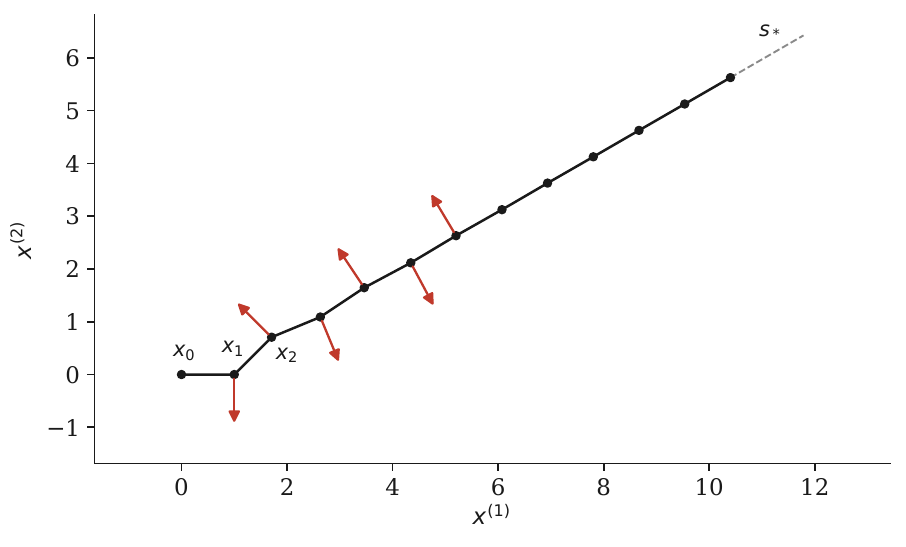}
\caption{The discrete orbit polyline $x_{k+1}=x_k+s_k$ for the first twelve iterations, illustrating the alternating and halving turns $\psi_{k+1}=-\frac12\psi_k$. Red arrows represent the unit gradient vectors $g_k = (-1)^k J s_{k-1}$ at vertices $x_1,\dots,x_6$, highlighting the exact line search orthogonality $\ip{g_{k+1}}{s_k}=0$ and transverse flapping. The dashed line marks the limiting ray with direction $s_* = (\frac{\sqrt3}{2}, \frac12)$.}
\label{fig:orbit}
\end{figure}

Geometrically, the trajectory alternates search directions while halving the turn angle at each step ($\psi_{k+1} = -\frac12\psi_k$), summing to the limiting direction $s_*$ (Figure~\ref{fig:orbit}). The secant decomposition and turn contraction are formalized next:

\begin{lemma}[Secant decomposition and turn contraction]
\label{lem:angles}
The trajectory \eqref{eq:orbit} and quantities \eqref{eq:derived} satisfy:
\begin{enumerate}
\item[(i)] \emph{Secant decomposition.} For every $k \ge 1$,
\begin{equation}
\label{eq:yk-geom}
y_k = (-1)^{k+1}J(s_k + s_{k-1}), \qquad \ip{y_k}{s_k + s_{k-1}} = 0.
\end{equation}
\item[(ii)] \emph{Turn contraction and step alignment.} For every $k \ge 1$, the signed turns alternate in sign and halve in magnitude:
\begin{equation}
\label{eq:turns}
\psi_k = (-1)^{k+1}\frac{\pi}{2^{k+1}}, \qquad \psi_{k+1} = -\tfrac12\psi_k.
\end{equation}
Moreover, the sum of consecutive search directions points along the next step:
\begin{equation}
\label{eq:bisect-vec}
s_k + s_{k-1} = 2\cos\left(\frac{\psi_k}{2}\right) s_{k+1}, \qquad \text{with } \cos\left(\frac{\psi_k}{2}\right) \ge \cos\frac{\pi}{8} > 0.
\end{equation}
\end{enumerate}
\end{lemma}

\begin{proof}
\emph{(i).} For $k \ge 1$, by \eqref{eq:orbit},
\[
y_k = g_{k+1} - g_k = (-1)^{k+1}Js_k - (-1)^kJs_{k-1} = (-1)^{k+1}J(s_k + s_{k-1}).
\]
Taking the inner product of both sides with $s_k + s_{k-1}$ gives $\ip{y_k}{s_k + s_{k-1}} = 0$.

\emph{(ii).} For $k \ge 1$, subtracting consecutive headings in \eqref{eq:orbit} gives
\[
\psi_k = \Theta_k - \Theta_{k-1} = (-1)^{k+1}\left(\frac{\pi}{3 \cdot 2^{k+1}} + \frac{\pi}{3 \cdot 2^k}\right) = (-1)^{k+1}\frac{\pi}{2^{k+1}},
\]
whence $\psi_{k+1} = -\tfrac12\psi_k$ and $\Theta_{k+1} = \frac12(\Theta_{k-1} + \Theta_k)$.
Therefore,
\begin{align*}
s_k + s_{k-1}
&= \bigl(\cos\Theta_k + \cos\Theta_{k-1},\; \sin\Theta_k + \sin\Theta_{k-1}\bigr)\\
&= 2\cos\left(\frac{\Theta_k - \Theta_{k-1}}{2}\right)\left(\cos\frac{\Theta_k + \Theta_{k-1}}{2},\; \sin\frac{\Theta_k + \Theta_{k-1}}{2}\right)\\
&= 2\cos\left(\frac{\psi_k}{2}\right) s_{k+1}.
\end{align*}
Since $|\psi_k| = \pi/2^{k+1} \le \pi/4$, we have $\cos(\psi_k/2) \ge \cos(\pi/8) > 0$.
\end{proof}

\subsection{Properties of the orbit and Class~\texorpdfstring{$\mathcal C$}{C} reproduction}

The next lemma collects the core properties of the orbit used throughout the paper:

\begin{lemma}[Properties of the orbit]
\label{lem:orbit}
For every $k \ge 0$, the sequence defined by \eqref{eq:orbit}--\eqref{eq:derived} satisfies:
\begin{enumerate}
\item[(i)] \emph{Exact steps:} $\ip{g_{k+1}}{s_k} = 0$.
\item[(ii)] \emph{Descent and curvature:}
\begin{equation}
\label{eq:costhetak}
\cos\theta_k = \ip{-g_k}{s_k} = -\eta_k = \sin\left(\frac{\pi}{2^{k+1}}\right) > 0;
\end{equation}
in particular $\eta_k < 0$, and
\begin{equation}
\label{eq:sigmak}
\sigma_k = \ip{s_k}{y_k} = -\eta_k = \cos\theta_k > 0.
\end{equation}
\item[(iii)] \emph{Conjugacy and descent of the next step:} $\ip{s_{k+1}}{y_k} = 0$ and $\ip{s_{k+1}}{g_{k+1}} < 0$; that is, $d = s_{k+1}$ satisfies the two Class~$\mathcal C$ requirements \eqref{eq:classC} at index $k$.
\end{enumerate}
\end{lemma}

\begin{proof}
\emph{(i).} By \eqref{eq:orbit},
$\ip{g_{k+1}}{s_k} = (-1)^{k+1}\ip{Js_k}{s_k} = 0$.

\emph{(ii).} Both $-g_k$ and $s_k$ are unit vectors, so $\cos\theta_k = \ip{-g_k}{s_k}$.
For $k = 0$, $\cos\theta_0 = \ip{-g_0}{s_0} = \ip{s_0}{s_0} = 1 = \sin(\pi/2)$.
For $k \ge 1$, $g_k = (-1)^k J s_{k-1}$, so by \eqref{eq:turns}, using the polar representations of $s_k, s_{k-1}$,
\begin{align*}
\cos\theta_k = \ip{-g_k}{s_k}
&= (-1)^{k+1}\ip{Js_{k-1}}{s_k} = (-1)^{k+1}\sin(\Theta_k - \Theta_{k-1})\\
&= (-1)^{k+1}\sin\psi_k = \sin\left(\frac{\pi}{2^{k+1}}\right),
\end{align*}
which is positive because $0 < \pi/2^{k+1} \le \pi/2$. Hence $\eta_k = \ip{g_k}{s_k} = -\cos\theta_k < 0$ for all $k \ge 0$.
By (i), $\sigma_k = \ip{s_k}{g_{k+1}} - \ip{s_k}{g_k} = 0 - \eta_k = \cos\theta_k > 0$.

\emph{(iii).} Strict descent is item (ii) read at index $k+1$: $\ip{s_{k+1}}{g_{k+1}} = \eta_{k+1} = -\sin(\pi/2^{k+2}) < 0$.

For conjugacy: at $k = 0$, \eqref{eq:orbit} gives $s_1 = (1/\sqrt2, 1/\sqrt2)$ and $y_0 = g_1 - g_0 = (1, -1)$, giving $\ip{s_1}{y_0} = 0$ directly.
For $k \ge 1$, by the step alignment formula \eqref{eq:bisect-vec} of Lemma~\ref{lem:angles},
\[
s_{k+1} = \frac{1}{2\cos(\psi_k/2)}(s_k + s_{k-1}).
\]
Because $\ip{y_k}{s_k + s_{k-1}} = 0$ by \eqref{eq:yk-geom}, it follows immediately that
\[
\ip{s_{k+1}}{y_k} = \frac{1}{2\cos(\psi_k/2)}\ip{y_k}{s_k + s_{k-1}} = 0.
\]
\end{proof}

It remains to explain in what sense the orbit is \emph{forced} on a method of class $\mathcal C$. The statement below is conditional and concerns a single step; the induction that turns it into a statement about entire trajectories needs the line search to return the designed steps, and is therefore deferred to the proof of Theorem~\ref{thm:main}.

\begin{lemma}[Heading lock]
\label{lem:heading}
Fix $k \ge 0$, and let a method of class $\mathcal C$ be applied to a function $f$ with $\nabla f(x_j) = g_j$ for $0 \le j \le k+1$. Assume that the method has reproduced the orbit through index $k$: it has visited the vertices $x_0, \dots, x_{k+1}$ of \eqref{eq:orbit}, using search directions $d_j \in \R_{>0}\,s_j$ for $0 \le j \le k$. Then the step just taken satisfies $\ip{\nabla f(x_{k+1})}{d_k} = 0$, the condition that triggers \eqref{eq:classC}, and the next direction produced by the method satisfies
\[
d_{k+1} \in \R_{>0}\,s_{k+1}.
\]
In particular, all methods of class $\mathcal C$ agree on the heading $s_{k+1}$; they may differ only in the length $\|d_{k+1}\|$.
\end{lemma}

\begin{proof}
Write $d_k = \lambda s_k$ with $\lambda > 0$. By Lemma~\ref{lem:orbit}(i),
\[
\ip{\nabla f(x_{k+1})}{d_k} = \lambda \ip{g_{k+1}}{s_k} = 0,
\]
so the defining orthogonality condition of class $\mathcal C$ is satisfied at index $k$, and the method must return a direction with $\ip{d_{k+1}}{y_k} = 0$ and $\ip{d_{k+1}}{g_{k+1}} < 0$.
By Lemma~\ref{lem:orbit}(ii), $\sigma_k = \ip{s_k}{y_k} > 0$, so $y_k \ne 0$. Moreover, $y_k \not\parallel g_{k+1}$ since $\ip{s_k}{y_k} > 0$ while $\ip{s_k}{g_{k+1}} = 0$.
Hence the first condition is a genuine linear constraint distinct from $g_{k+1}^\perp$; Lemma~\ref{lem:unique} therefore applies and the two conditions together determine a unique ray in $\R^2$. By Lemma~\ref{lem:orbit}(iii), $s_{k+1}$ satisfies both conditions, so that ray is $\R_{>0}\,s_{k+1}$, and $d_{k+1}$ lies on it. The data $g_{k+1}$ and $y_k$ entering \eqref{eq:classC} are determined by the orbit alone, so the conclusion is the same for every method of the class.
\end{proof}

\subsection{Polyline geometry and non-self-intersection}
\label{sec:geom}

If the discrete trajectory crossed itself, any single-valued objective function would be forced to carry two incompatible gradient prescriptions at the crossing point, making it impossible for the orbit to be realized as the trajectory of any $C^1$ function at all.
Simplicity of the polyline $P$ is therefore an essential feasibility condition for the construction to exist.
This subsection shows that $P$ avoids itself with a quantitative margin: non-adjacent segments stay at least $1$ apart, matching the step length itself.
To state this precisely, call $[x_k, x_{k+1}]$ \emph{segment $k$}; call segments $k$ and $j$ \emph{non-adjacent} when they share no endpoint ($|k-j| \ge 2$); and call the integer $|k-j|$ their \emph{index gap}.

From \eqref{eq:orbit}, every heading satisfies $|\Theta_k - \pi/6| = \frac{\pi}{3 \cdot 2^{k+1}} \le \frac{\pi}{6}$, so all headings remain confined to the acute cone $\Theta_k \in [0, \pi/3]$ for all $k \ge 0$.
Consequently, the angle between any two search directions satisfies $|\Theta_i - \Theta_j| \le \pi/3$, yielding the acute-cone property:
\begin{equation}
\label{eq:acute-cone}
\ip{s_i}{s_j} = \cos(\Theta_i - \Theta_j) \ge \cos\left(\frac{\pi}{3}\right) = \frac12 > 0 \qquad \text{for all } i, j \ge 0.
\end{equation}

\begin{lemma}[Segment separation]
\label{lem:geom}
Every pair of non-adjacent segments of the polyline $P = \bigcup_{k\ge0}[x_k, x_{k+1}]$ has Euclidean distance at least $1$.
\end{lemma}

\begin{proof}
Let $k \ge 0$ and $g \ge 2$.
Any point $z_1$ on segment $k$ and $z_2$ on segment $k+g$ can be written as
\[
z_1 = x_{k+1} - as_k, \qquad z_2 = x_{k+g} + bs_{k+g}
\]
for some $a, b \in [0, 1]$.
Since $x_{k+g} - x_{k+1} = \sum_{m=1}^{g-1} s_{k+m}$, the displacement vector is
\[
z_2 - z_1 = as_k + s_{k+1} + \sum_{m=2}^{g-1} s_{k+m} + bs_{k+g}
\]
(where the sum over $m$ is empty when $g = 2$).
Every coefficient in this expansion is non-negative, and the coefficient of $s_{k+1}$ is $1$.
Pairing with the unit vector $s_{k+1}$ and applying \eqref{eq:acute-cone} to each term,
\[
\ip{z_2 - z_1}{s_{k+1}} = 1 + a\ip{s_k}{s_{k+1}} + \sum_{m=2}^{g-1} \ip{s_{k+m}}{s_{k+1}} + b\ip{s_{k+g}}{s_{k+1}} \ge 1.
\]
By the Cauchy--Schwarz inequality,
\[
\|z_2 - z_1\| \ge \ip{z_2 - z_1}{s_{k+1}} \ge 1.
\]
Thus, every pair of non-adjacent segments has Euclidean distance $\ge 1$.
\end{proof}

\begin{corollary}[Non-self-intersection]
\label{cor:nonself}
The polyline $P = \bigcup_{k\ge0}[x_k, x_{k+1}]$ does not self-intersect; that is, distinct segments intersect if and only if they are adjacent, in which case they meet only at their common vertex.
\end{corollary}

\begin{proof}
Non-adjacent segments satisfy $\|z_1 - z_2\| \ge 1 > 0$ by Lemma~\ref{lem:geom}, and so do not intersect.
For adjacent segments $[x_k, x_{k+1}]$ and $[x_{k+1}, x_{k+2}]$, the turn angle between their directions satisfies $0 < |\psi_{k+1}| = \pi/2^{k+2} \le \pi/4 < \pi$ by \eqref{eq:turns}, so the two segments lie on distinct, non-parallel lines in $\R^2$.
Because two non-parallel lines intersect at exactly one point, the adjacent segments intersect only at their common endpoint $x_{k+1}$.
\end{proof}

While the subsequent domain assembly relies formally on the quantitative separation of strips and disks (Lemma~\ref{lem:seam}), Corollary~\ref{cor:nonself} establishes why the construction is feasible: the polyline is simple, so the discrete gradient assignments along distinct steps never conflict.

\subsection{Geometry of the tubular neighborhood}
\label{sec:orbit-tube}

Just as Section~\ref{sec:geom} established that the one-dimensional polyline $P$ does not self-intersect, we now construct a thickened two-dimensional domain $\mathcal U \subset \R^2$ around $P$ and verify its geometric consistency.
Let $\delta$ be a scale parameter satisfying
\begin{equation}
\label{eq:delta}
0 < \delta \le \frac{1}{20}.
\end{equation}
The width $\delta$ is a free parameter; all geometric separation bounds in this section hold uniformly for every $\delta \in (0, 1/20]$, and $\delta$ will be calibrated later in Section~\ref{sec:ls} to satisfy the line-search criteria.

Along segment $k$, let $\nu_k = J s_k$ be the unit normal to $s_k$, so that $\{s_k, \nu_k\}$ forms an orthonormal basis of $\R^2$.
In the associated Fermi coordinates $(\tau, n)$ centered at $x_k$, each point $x$ is represented as
\begin{equation}
\label{eq:fermi}
x = x_k + \tau s_k + n \nu_k, \qquad \tau = \ip{x-x_k}{s_k}, \quad n = \ip{x-x_k}{\nu_k}.
\end{equation}
The segment strip is defined by $\mathcal S_k = \{x_k + \tau s_k + n\nu_k : \tau \in [0, 1],\, |n| < \delta\}$ for $k \ge 1$.
To place the initial vertex $x_0$ safely in the interior of the domain, the initial strip $\mathcal S_0$ is extended backward by $\delta$ along $-s_0$:
\begin{equation}
\label{eq:S0}
\mathcal S_0 = \{x_0 + \tau s_0 + n\nu_0 : \tau \in (-\delta, 1],\, |n| < \delta\}.
\end{equation}
It is convenient to name the \emph{axis} of each strip: $\ell_j = [x_j, x_{j+1}]$ for $j \ge 1$, and $\ell_0 = [x_0 - \delta s_0, x_1]$ for the backward-extended initial strip.
Every point of $\mathcal S_j$ then lies at distance $|n_j| < \delta$ from $\ell_j$, for every $j \ge 0$.

To prepare for the smooth interpolation in Section~\ref{sec:f} (where the objective function will be governed by isotropic Taylor quadratics on vertex disks and collars, and by a strip template along the middle tubes), we describe the constituent pieces of the tubular neighborhood $\mathcal U$ (identifying points on each strip $\mathcal S_k$ with their Fermi coordinates $(\tau, n)$; see Figure~\ref{fig:tube}):
\begin{enumerate}
\item \emph{Vertex disks:} $\mathcal B_k = B(x_k, \delta)$ for each $k \ge 1$. (No vertex disk is needed at $x_0$, which lies inside the backward extension of $\mathcal S_0$.)
\item \emph{Middle tubes:} On each segment $k \ge 1$, $\mathcal T_k = (\delta, 1-\delta) \times (-\delta, \delta)$. For segment $0$, the extended middle tube is $\mathcal T_0 = (-\delta, 1-\delta) \times (-\delta, \delta)$.
\item \emph{Collars:} On each segment $k \ge 0$, the strip $\mathcal S_k$ contains an outgoing (or left) collar $[0, \delta] \times (-\delta, \delta)$ at $x_k$ and an incoming (or right) collar $[1-\delta, 1] \times (-\delta, \delta)$ at $x_{k+1}$.
At each vertex $x_k$ ($k \ge 1$), the meeting collar region $\mathcal C_k \subset \R^2$ connects the incoming collar of segment $k-1$ and the outgoing collar of segment $k$:
\begin{equation}
\label{eq:Ckdef}
\mathcal C_k = \{p \in \mathcal S_{k-1} : \tau_{k-1}(p) \ge 1-\delta\} \cup \{p \in \mathcal S_k : \tau_k(p) \le \delta\}.
\end{equation}
\end{enumerate}
The total tubular neighborhood of the polyline is the open domain
\begin{equation}
\label{eq:Udef}
\mathcal U = \bigcup_{k\ge 1}\mathcal B_k \cup \bigcup_{k\ge 0}\mathcal S_k = \bigcup_{k\ge 1}\mathcal B_k \cup \bigcup_{k\ge 0}\mathcal T_k \cup \bigcup_{k\ge 1}\mathcal C_k \subset \R^2.
\end{equation}
Each strip and each disk surrounds a set of vertices of $P$ (the strip $\mathcal S_j$ surrounds $x_j$ and $x_{j+1}$, the disk $\mathcal B_k$ surrounds $x_k$), and we call two of them \emph{adjacent} when these vertex sets meet, and \emph{non-adjacent} otherwise.

\begin{figure}[t]
\centering
\includegraphics[width=0.88\textwidth]{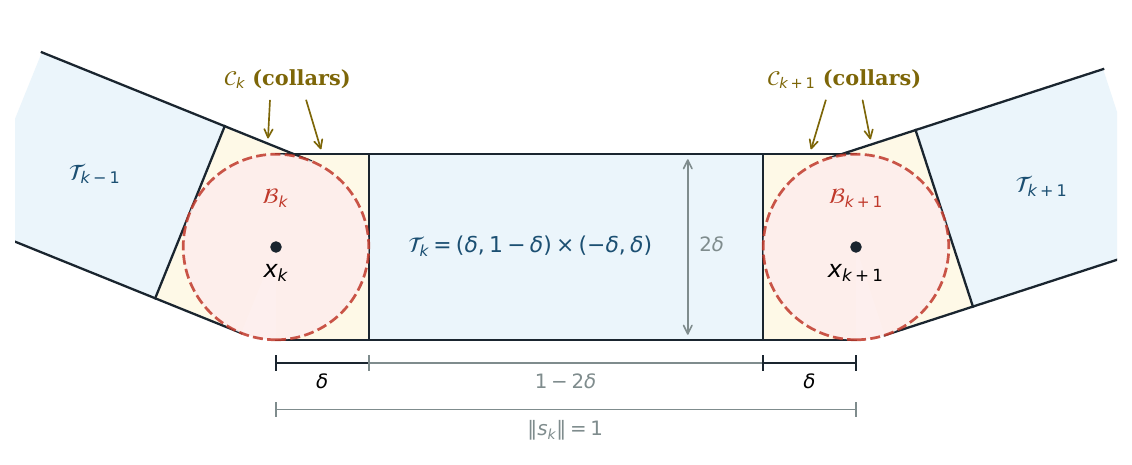}
\caption{Decomposition of the tubular neighborhood $\mathcal U$ along segment $k$: vertex disks $\mathcal B_k, \mathcal B_{k+1}$ of radius $\delta$, meeting collar pairs $\mathcal C_k, \mathcal C_{k+1}$, and the middle tube $\mathcal T_k = (\delta, 1-\delta) \times (-\delta, \delta)$. Adjacent strips $\mathcal S_{k-1}$ and $\mathcal S_k$ overlap exclusively within the collar regions, where both strip templates reduce identically to the vertex Taylor quadratic $Q_k$ (Lemma~\ref{lem:seam}).}
\label{fig:tube}
\end{figure}

\begin{lemma}[Overlaps of strips and disks]
\label{lem:seam}
Let $0 < \delta \le 1/20$.
\begin{enumerate}
\item[(i)] \emph{Non-adjacent strips and disks are separated.}
Non-adjacent strips and disks have disjoint closures:
\[
\overline{\mathcal S_i} \cap \overline{\mathcal S_j} = \emptyset \ \ (|i-j| \ge 2), \qquad
\overline{\mathcal B_k} \cap \overline{\mathcal S_j} = \emptyset \ \ (j \notin \{k-1, k\}), \qquad
\overline{\mathcal B_k} \cap \overline{\mathcal B_j} = \emptyset \ \ (j \ne k).
\]
Moreover every point of $\R^2$ has a neighborhood meeting at most two strips (necessarily consecutive) and at most one vertex disk.
\item[(ii)] \emph{Overlaps at a corner lie in the collar.}
For every $k \ge 1$, all overlaps among the adjacent strips $\mathcal S_{k-1}, \mathcal S_k$ and the vertex disk $\mathcal B_k$ lie in the open Fermi wedge, and hence in the collar $\mathcal C_k$:
\begin{equation}
\label{eq:cornerincl}
(\mathcal S_{k-1} \cap \mathcal S_k) \cup \bigl(\mathcal B_k \cap (\mathcal S_{k-1} \cup \mathcal S_k)\bigr) \subset \{\tau_{k-1}>1-\delta\}\cap\{\tau_k<\delta\}.
\end{equation}
Moreover, restricting $\mathcal C_k$ to either strip recovers precisely that strip's own collar:
\begin{equation}
\label{eq:Ckrest}
\mathcal C_k \cap \mathcal S_{k-1} = \{p \in \mathcal S_{k-1} : \tau_{k-1}(p) \ge 1-\delta\}, \qquad
\mathcal C_k \cap \mathcal S_k = \{p \in \mathcal S_k : \tau_k(p) \le \delta\}.
\end{equation}
\end{enumerate}
\end{lemma}

\begin{proof}
\emph{(i).}
Explicitly, two of these sets are non-adjacent when they are strips $\mathcal S_i, \mathcal S_j$ with $|i-j| \ge 2$, a disk $\mathcal B_k$ and a strip $\mathcal S_j$ with $j \notin \{k-1, k\}$, or two distinct disks.
Each of them lies within distance $\sqrt2\,\delta$ of its \emph{generator}: the unextended segment $[x_j, x_{j+1}]$ for the strip $\mathcal S_j$, and the vertex $x_k$ for the disk $\mathcal B_k$.
Indeed, a point of $\mathcal S_j$ is at orthogonal distance $|n_j| < \delta$ from segment $j$, the only exception being the backward extension of $\mathcal S_0$, whose points have closest point $x_0$ on segment $0$ at distance $\sqrt{\tau_0^2 + n_0^2} < \sqrt2\,\delta$; and $\mathcal B_k = B(x_k, \delta)$.
In each of the three configurations the generators are at Euclidean distance at least $1$:
\begin{itemize}
\item non-adjacent segments ($|i-j| \ge 2$), by Lemma~\ref{lem:geom};
\item the vertex $x_k$ and segment $j$ with $j \notin \{k-1, k\}$, since $x_k$ lies on segment $k-1$ (if $j \ge k+1$) or on segment $k$ (if $j \le k-2$), each of which is non-adjacent to segment $j$, so that $\operatorname{dist}(x_k, \text{segment } j) \ge 1$ by Lemma~\ref{lem:geom};
\item two distinct vertices $x_k, x_j$, either because $\|x_{k+1}-x_k\| = \|s_k\| = 1$ when $|k-j| = 1$, or by Lemma~\ref{lem:geom} when $|k-j| \ge 2$.
\end{itemize}
By the triangle inequality all three distances are therefore at least $1 - 2\sqrt2\,\delta \ge 1-3\delta \ge 17/20 > 0$, using $\delta \le 1/20$ from \eqref{eq:delta}.
Being at positive distance, the closures are disjoint; and no ball of radius $(1-3\delta)/2$ meets both members of such a pair, so every point of $\R^2$ has a neighborhood meeting at most two strips (necessarily consecutive) and at most one vertex disk.

\emph{(ii).}
We first show, for every $p \in \R^2$ and every $k \ge 1$,
\begin{equation}
\label{eq:corner}
\operatorname{dist}(p, \ell_{k-1}) < \delta \implies \tau_k(p) < \delta,
\qquad
\operatorname{dist}(p, \ell_k) < \delta \implies \tau_{k-1}(p) > 1-\delta.
\end{equation}
Let $q$ be the point of $\ell_{k-1}$ closest to $p$.
Since $\ell_{k-1}$ terminates at $x_k$, we have $q = x_k - b\,s_{k-1}$ with $b \ge 0$, while $\ip{s_{k-1}}{s_k} \ge 1/2 > 0$ by the acute-cone property \eqref{eq:acute-cone}; hence
\[
\tau_k(p) = \ip{p-x_k}{s_k} = \ip{p-q}{s_k} - b\ip{s_{k-1}}{s_k} \le \|p-q\| = \operatorname{dist}(p, \ell_{k-1}) < \delta.
\]
Exchanging the roles of the two segments, with $q' = x_k + a\,s_k$ ($a \ge 0$) the point of $\ell_k$ closest to $p$,
\[
1 - \tau_{k-1}(p) = \ip{x_k-p}{s_{k-1}} = \ip{q'-p}{s_{k-1}} - a\ip{s_k}{s_{k-1}} \le \|q'-p\| = \operatorname{dist}(p, \ell_k) < \delta,
\]
which is the second implication in \eqref{eq:corner}.
Note that no upper bound on $\delta$ is used here: the corner geometry is scale free.

Now, points of $\mathcal S_{k-1}$ lie within $\delta$ of $\ell_{k-1}$, points of $\mathcal S_k$ lie within $\delta$ of $\ell_k$, and points of $\mathcal B_k = B(x_k, \delta)$ lie within $\delta$ of $x_k \in \ell_{k-1} \cap \ell_k$, so both implications of \eqref{eq:corner} apply on $\mathcal B_k$ and give $\tau_{k-1}>1-\delta$ and $\tau_k<\delta$.
By \eqref{eq:corner}, points in $\mathcal S_{k-1} \cap \mathcal S_k$ likewise have both $\tau_{k-1} > 1-\delta$ and $\tau_k < \delta$. This is \eqref{eq:cornerincl}.
A point of $\mathcal S_{k-1}$ with $\tau_{k-1}>1-\delta$ lies in $\mathcal C_k$ by \eqref{eq:Ckdef}, and likewise for $\mathcal S_k$, so the overlaps also lie in $\mathcal C_k$.
Finally, intersecting \eqref{eq:Ckdef} with $\mathcal S_{k-1}$ gives $\{p \in \mathcal S_{k-1} : \tau_{k-1} \ge 1-\delta\} \cup (\{p \in \mathcal S_k : \tau_k \le \delta\} \cap \mathcal S_{k-1})$; since the second term lies in $\mathcal S_k \cap \mathcal S_{k-1}$, every point in it satisfies $\tau_{k-1} > 1-\delta$ by \eqref{eq:corner}, so it is absorbed into the first term. This proves the first equality in \eqref{eq:Ckrest}, and the second follows symmetrically.
\end{proof}

\section{The smooth interpolant}
\label{sec:f}

We now realize the discrete orbit \eqref{eq:orbit} as the vertex
sequence of an actual $C^\infty$ function $f$ on all of $\R^2$, bounded
below and with Lipschitz gradient.
Near each vertex, $f$ matches the prescribed value and gradient via isotropic Taylor quadratics; along the segment strips, $f$ is given by a strip template, quadratic in the transverse variable with smooth axial coefficients; and outside the tubular neighborhood $\mathcal U$, $f$ transitions smoothly to a finite constant.

\subsection{The strip template}
\label{sec:template}

The discrete orbit fixes the vertices $x_k$ and gradients $g_k$.
At each vertex $x_{k+1}$ ($k \ge 0$), the incoming segment direction $s_k$ and the outgoing direction $s_{k+1}$ are non-collinear ($s_k \not\parallel s_{k+1}$, since the turning angle $\psi_{k+1} \ne 0$ by \eqref{eq:turns}).
To smoothly patch the local coordinates across the corner without requiring higher-order tensorial matching between the rotated frames $\{s_k, \nu_k\}$ and $\{s_{k+1}, \nu_{k+1}\}$, we therefore require that the Hessian at $x_{k+1}$ be rotationally invariant (isotropic), $\nabla^2 f(x_{k+1}) = \alpha_{k+1} I$.
Since $\|s_k\|=1$, the directional curvature along the incoming step is $s_k^{\mathsf T}(\alpha_{k+1} I)s_k = \alpha_{k+1}\|s_k\|^2 = \alpha_{k+1}$.
Equating this directional curvature to the secant curvature $\sigma_k = \ip{s_k}{y_k} = \cos\theta_k$ established in \eqref{eq:sigmak}, we set
\begin{equation}
\label{eq:alphak}
\alpha_{k+1} = \sigma_k = \cos\theta_k > 0 \quad (k \ge 0),
\end{equation}
so that $\nabla^2 f(x_{k+1}) = \alpha_{k+1} I$ is positive definite. At the initial vertex $x_0$ there is no incoming step; we set $f_0 = 0$ and $\alpha_0 = 0$, and note $\alpha_1 = \cos\theta_0 = \sin(\pi/2) = 1$.
On the open disk $\mathcal B_k = B(x_k, \delta)$ ($k \ge 1$), the vertex quadratic is
\begin{equation}
\label{eq:Qk}
Q_k(x) = f_k + \ip{g_k}{x-x_k} + \frac{\alpha_k}{2}\|x-x_k\|^2.
\end{equation}
At $x_0$, the linear Taylor polynomial is $Q_0(x) = \ip{g_0}{x-x_0}$ (with $f_0 = 0$ and $\alpha_0 = 0$).

In the orthonormal Fermi frame $\{s_k, \nu_k\}$ on the segment strip $\mathcal S_k$ defined in \eqref{eq:fermi}, the endpoints $x_k$ and $x_{k+1} = x_k + s_k$ correspond to $(\tau, n) = (0, 0)$ and $(1, 0)$ respectively.
Because $\{s_k, \nu_k\}$ is orthonormal, the distance squares from the endpoints are $\|x-x_k\|^2 = \tau^2 + n^2$ and $\|x-x_{k+1}\|^2 = (\tau-1)^2 + n^2$.
Substituting into \eqref{eq:Qk}, the endpoint quadratics expand as
\begin{align}
Q_k(\tau, n) &= \bigl(f_k + \tau\eta_k + \tfrac12\alpha_k\tau^2\bigr) + n\ip{g_k}{\nu_k} + \tfrac12\alpha_k n^2, \label{eq:aleft}\\
Q_{k+1}(\tau, n) &= \bigl(f_{k+1} + \tfrac12\alpha_{k+1}(\tau-1)^2\bigr) + n\ip{g_{k+1}}{\nu_k} + \tfrac12\alpha_{k+1} n^2, \label{eq:aright}
\end{align}
where $\eta_k = \ip{g_k}{s_k} = -\cos\theta_k$ by \eqref{eq:costhetak}, and the $\tau$-linear term in $Q_{k+1}$ vanishes since $\ip{g_{k+1}}{s_k} = 0$ (Lemma~\ref{lem:orbit}(i)).

Both endpoint quadratics \eqref{eq:aleft} and \eqref{eq:aright} have the same form in the transverse displacement $n$: an axial profile, a linear tilt, and a quadratic term whose coefficient is half the axial second derivative, reflecting the isotropic Hessian at the vertices. We therefore interpolate across the strip by a template of this form,
\begin{equation}
\label{eq:template}
F_k(\tau, n) = a(\tau) + b(\tau)n + \tfrac12 a''(\tau)n^2.
\end{equation}
In the orthonormal frame $\{s_k, \nu_k\}$, the on-axis Hessian along the centerline $n=0$ is
\begin{equation}
\label{eq:Hess0}
\nabla^2 F_k(\tau, 0) = \begin{pmatrix} a''(\tau) & b'(\tau) \\ b'(\tau) & a''(\tau) \end{pmatrix} = a''(\tau) I + b'(\tau)\begin{pmatrix} 0 & 1 \\ 1 & 0 \end{pmatrix}.
\end{equation}
Whenever $b'(\tau) = 0$, this Hessian is the isotropic matrix $a''(\tau) I$, as at the vertices.

To match those vertex Hessians on the collars we take $a''\equiv\alpha_k$ on a left collar and $a''\equiv\alpha_{k+1}$ on a right collar, and we take $b$ constant on each collar (so $b'=0$ there), interpolating the transverse slopes $\ip{g_k}{\nu_k}$ and $\ip{g_{k+1}}{\nu_k}$. Between the collars we place two spikes, leaving a central gap with $a''\equiv0$. We fix the total mass of $a''$ at $\alpha_{k+1}$, so that the axial slope vanishes at $\tau=1$; the leftover mass after the collars is carried by the two spikes, and shifting it from one spike to the other moves the center of mass of $a''$, which will control the Armijo ratio. This profile is shown in Figure~\ref{fig:spike}.

To implement this shape in $C^\infty$ we use two elementary building blocks: the standard smooth step $S$ and localized window $\Pi$.
Let $\chi(t) = 0$ for $t \le 0$ and $\chi(t) = \exp(-1/t)$ for $t > 0$, and set
\begin{equation}
\label{eq:Sdef}
S(u) = \frac{\chi(u)}{\chi(u) + \chi(1-u)} \qquad (u \in \R).
\end{equation}
For breakpoints $t_0 < t_1$ and $t_0 < t_1 \le t_2 < t_3$, set
\begin{align}
\label{eq:Sparam}
S(t; t_0, t_1) &:= S\Bigl(\frac{t-t_0}{t_1-t_0}\Bigr), \\
\label{eq:Pidef}
\Pi(t; t_0, t_1, t_2, t_3) &:= S(t; t_0, t_1) - S(t; t_2, t_3).
\end{align}
Thus $S(\cdot; t_0, t_1)$ transitions from $0$ on $(-\infty, t_0]$ to $1$ on $[t_1, \infty)$, and $\Pi$ is a window equal to $1$ on $[t_1, t_2]$ and vanishing outside $(t_0, t_3)$; when $t_1 = t_2$ it degenerates to a localized spike with peak value $1$ at $t = t_1$.
Their elementary analytical properties (smoothness, monotonicity, flat boundary derivatives of all orders, total mass, and center of mass) are established in Appendix~\ref{sec:app-blocks} (Lemmas~\ref{lem:S}, \ref{lem:trans}, and~\ref{lem:window}).

The transverse tilt is
\begin{equation}
\label{eq:btrans}
b(\tau) = \ip{g_k}{\nu_k} + \bigl(\ip{g_{k+1}}{\nu_k} - \ip{g_k}{\nu_k}\bigr)
S(\tau; \delta, 1-\delta).
\end{equation}
The axial curvature is assembled from two unilateral profiles, each a collar cap plus a spike,
\begin{align}
\Lambda_k^-(\tau) &:= \alpha_k\bigl(1 - S(\tau; \delta, 2\delta)\bigr) + \lambda_k M_k\,\Pi(\tau; \delta, 2\delta, 2\delta, 3\delta), \label{eq:Lambdaminus} \\
\Lambda_k^+(\tau) &:= \alpha_{k+1}\bigl(1 - S(\tau; \delta, 2\delta)\bigr) + (1-\lambda_k) M_k\,\Pi(\tau; \delta, 2\delta, 2\delta, 3\delta), \label{eq:Lambdaplus}
\end{align}
by
\begin{equation}
\label{eq:Adef}
a''(\tau) = \Lambda_k^-(\tau) + \Lambda_k^+(1-\tau),
\end{equation}
with surplus height
\begin{equation}
\label{eq:Msum}
M_k = \frac{\alpha_{k+1}}{\delta} - \tfrac32(\alpha_k + \alpha_{k+1})
\end{equation}
split by $\lambda_k$ between the two peaks.

The axial profile $a$ of \eqref{eq:template} is recovered from $a''$ by integrating twice. Setting $f_0=0$, let $a_k$ be the second antiderivative of \eqref{eq:Adef} fixed by the left-vertex data, and let the next height be its endpoint value:
\begin{equation}
\label{eq:adef}
a_k(0)=f_k, \qquad a_k'(0)=\eta_k, \qquad f_{k+1}:=a_k(1) \qquad (k\ge0);
\end{equation}
we drop the index $k$ when the segment is fixed. Successive segments are coupled only through the constant $f_k$, so the descent $f_k-f_{k+1}$ depends on $\lambda_k$ alone.

For the initial strip $\mathcal S_0$ in \eqref{eq:S0}, the same formulas give $a''\equiv0$ and $b\equiv0$ on $(-\delta,\delta]$, and
\begin{equation}
\label{eq:Q0collar}
F_0(\tau, n) = -\tau = Q_0(\tau, n) \quad\text{on } (-\delta, \delta] \times (-\delta, \delta).
\end{equation}

Section~\ref{sec:props} verifies four properties of this $F_k$, in this order:
\begin{itemize}
\item \emph{Axial curvature} (Lemma~\ref{lem:admissible}): $a''$ matches the vertex curvatures on the collars, carries total mass $\alpha_{k+1}$, and has a tunable center of mass.
\item \emph{Axial profile} (Lemma~\ref{lem:profile}): the integral $a$ decreases to a stationary point at $\tau=1$, with Armijo ratio equal to that center of mass.
\item \emph{Collar agreement} (Lemma~\ref{lem:collar}): $F_k$ coincides with $Q_k$ and $Q_{k+1}$ on the collars.
\item \emph{Hessian bound} (Lemma~\ref{lem:hessU}): $\nabla^2 F_k$ is controlled by $\delta$, uniformly in $k$, so that $\nabla f$ will be Lipschitz.
\end{itemize}

\begin{figure}[t]
\centering
\includegraphics[width=0.76\textwidth]{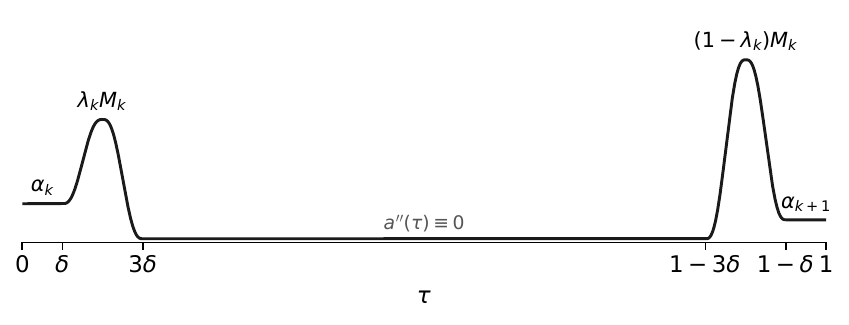}
\caption{The axial curvature profile $a''(\tau) = \Lambda_k^-(\tau) + \Lambda_k^+(1-\tau)$ along a typical segment ($k\ge1$, $\delta=1/20$) defined in \eqref{eq:Adef}, shown with split parameter $\lambda_k=0.4$. The flat collar caps $[0, \delta]$ and $[1-\delta, 1]$ preserve the isotropic vertex curvatures $\alpha_k$ and $\alpha_{k+1}$ identically. In the interior, two smooth spike windows of heights $\lambda_k M_k$ (at $\tau=2\delta$) and $(1-\lambda_k)M_k$ (at $\tau=1-2\delta$) share the surplus height $M_k$ between them, calibrating the Armijo ratio, leaving a central zero-curvature gap $a'' \equiv 0$ on $[3\delta, 1-3\delta]$.}
\label{fig:spike}
\end{figure}

\subsection{Properties of the strip template}
\label{sec:props}

We now prove the four properties listed in Section~\ref{sec:template}.
All estimates in this section hold for every $\delta\in(0,1/20]$ and every $\lambda_k\in[0,1]$, with constants independent of $k$ and $\lambda_k$; these two parameters are fixed in Section~\ref{sec:ls}.

\begin{lemma}[Axial curvature]
\label{lem:admissible}
Fix $k\ge0$ and, for each $\lambda\in[0,1]$ (suppressing the index $k$ on $\lambda$ for brevity), let $a''_\lambda$ denote the axial curvature \eqref{eq:Adef} with split parameter $\lambda$. Then:
\begin{enumerate}
\item[(i)] \emph{Collars and mass.} The axial curvature $a''_\lambda$ satisfies $a''_\lambda\ge0$ on $[0,1]$, $a''_\lambda\equiv\alpha_k$ on $[0,\delta]$, $a''_\lambda\equiv\alpha_{k+1}$ on $[1-\delta,1]$, and $\int_0^1 a''_\lambda=\alpha_{k+1}$.
\item[(ii)] \emph{Center-of-mass sweep.} The normalized center of mass
\[
r(\lambda) := \frac1{\alpha_{k+1}}\int_0^1\tau\,a''_\lambda(\tau)\,d\tau
\]
is affine and strictly decreasing, hence maps $[0,1]$ onto $[r(1), r(0)]$, where
\[
r(1)<6\delta \qquad\text{and}\qquad r(0)>1-6\delta.
\]
\end{enumerate}
\end{lemma}

\begin{proof}
By \eqref{eq:Lambdaminus}--\eqref{eq:Adef}, $a''_\lambda(\tau)=\Lambda_k^-(\tau) + \Lambda_k^+(1-\tau)$. Throughout, $\alpha_{k+1}=\cos\theta_k>0$ by \eqref{eq:alphak}, while $\alpha_k=\cos\theta_{k-1}>0$ for $k\ge1$ and $\alpha_0=0$; hence $\alpha_k\ge0$ and $\alpha_{k+1}>0$.

\emph{(i).} For $k\ge1$, Lemma~\ref{lem:orbit}(ii) gives
\begin{equation}
\label{eq:rk}
\frac{\alpha_k}{\alpha_{k+1}} = \frac{\cos\theta_{k-1}}{\cos\theta_k}
= \frac{\sin(\pi/2^k)}{\sin(\pi/2^{k+1})}
= 2\cos\bigl(\pi/2^{k+1}\bigr) < 2,
\end{equation}
while $\alpha_0=0$. Thus $\alpha_k<2\alpha_{k+1}$ for all $k\ge0$, and with $\delta \le 1/20$ by \eqref{eq:delta} the surplus height is strictly positive:
\[
M_k > \alpha_{k+1}(\delta^{-1} - \tfrac92) \ge \tfrac{31}{2}\alpha_{k+1} > 0.
\]
With $\alpha_k, \alpha_{k+1}, M_k \ge 0$ and $\lambda \in [0, 1]$, each coefficient in $\Lambda_k^-$ and $\Lambda_k^+$ is non-negative. Because both $1-S \in [0, 1]$ (Lemma~\ref{lem:trans}(i)) and $\Pi \in [0, 1]$ (Lemma~\ref{lem:window}(i)), each profile is non-negative, so $a''_\lambda \ge 0$.

For $\delta\le 1/20$, the supports of $\Lambda_k^-(\tau)$ and $\Lambda_k^+(1-\tau)$ lie in $[0,3\delta]$ and $[1-3\delta,1]$, which are disjoint:
\begin{equation}
\label{eq:disjoint}
3\delta \le \frac{3}{20} < \frac{17}{20} \le 1-3\delta.
\end{equation}
On $[0, \delta]$, $1-S(\tau; \delta, 2\delta) \equiv 1$ and $\Pi(\tau; \delta, 2\delta, 2\delta, 3\delta) \equiv 0$ by Lemma~\ref{lem:trans}(i) and Lemma~\ref{lem:window}(i), so $\Lambda_k^-(\tau) \equiv \alpha_k$. By the support separation \eqref{eq:disjoint}, $\Lambda_k^+(1-\tau) \equiv 0$ on $[0, \delta]$, so $a''_\lambda \equiv \alpha_k$ on $[0, \delta]$. Symmetrically, on $[1-\delta, 1]$, $\Lambda_k^-(\tau) \equiv 0$ and $\Lambda_k^+(1-\tau) \equiv \alpha_{k+1}$, so $a''_\lambda \equiv \alpha_{k+1}$ on $[1-\delta, 1]$.

By Lemma~\ref{lem:trans}(iv), each reversed transition cap $1-S(\cdot; \delta, 2\delta)$ has integral $\frac32\delta$, while by Lemma~\ref{lem:window}(iv) each spike window has mass $\delta$.
Integrating \eqref{eq:Adef} gives
\begin{equation}
\label{eq:target}
\int_0^1 a''_\lambda(\tau)\,d\tau = \tfrac32(\alpha_k + \alpha_{k+1})\delta + [\lambda + (1-\lambda)]M_k\delta = \tfrac32(\alpha_k + \alpha_{k+1})\delta + M_k\delta \overset{\eqref{eq:Msum}}{=} \alpha_{k+1},
\end{equation}
independently of $\lambda$.

\emph{(ii).}
Because $\lambda$ enters \eqref{eq:Lambdaminus}--\eqref{eq:Adef} affinely via the spike heights $\lambda M_k$ and $(1-\lambda)M_k$, we have $a''_\lambda(\tau) = \lambda a''_1(\tau) + (1-\lambda) a''_0(\tau)$ pointwise on $[0, 1]$.
Integrating against $\tau$ shows that $r(\lambda) = \lambda r(1) + (1-\lambda) r(0)$ is affine.
Strict decrease then follows directly from the endpoint estimates below, which give $r(1) < 6\delta \le 3/10 < 7/10 \le 1-6\delta < r(0)$ for $\delta \le 1/20$, so that the slope $r(1) - r(0)$ is strictly negative.

To establish these range bounds, crude support bounds suffice:

\emph{Case $\lambda=1$ (all surplus on the left spike):}
The right spike receives no surplus ($\lambda=1$), so the mass of $a''_1$ on $[1-3\delta, 1]$ reduces to that of the collar cap, which by Lemma~\ref{lem:trans}(iv) is $\int_{1-3\delta}^1 a''_1 = \tfrac32\alpha_{k+1}\delta$. Using $\tau \le 3\delta$ on the left support $[0, 3\delta]$ and $\tau \le 1$ on the right, and bounding the left mass by the total mass $\alpha_{k+1}$,
\[
\int_0^1 \tau\,a''_1(\tau)\,d\tau \le 3\delta\int_0^{3\delta} a''_1 + 1\cdot\int_{1-3\delta}^1 a''_1 \le 3\delta\alpha_{k+1} + \tfrac32\alpha_{k+1}\delta = \tfrac92\alpha_{k+1}\delta < 6\alpha_{k+1}\delta.
\]
Dividing by $\alpha_{k+1}$ gives $r(1) < 6\delta$.

\emph{Case $\lambda=0$ (all surplus on the right spike):}
The left spike receives no surplus ($\lambda=0$), so the mass of $a''_0$ on $[0, 3\delta]$ reduces to that of the collar cap, $\int_0^{3\delta} a''_0(\tau)\,d\tau = \tfrac32\alpha_k\delta$. For $k=0$ this mass vanishes since $\alpha_0=0$; for $k\ge1$, $\alpha_k < 2\alpha_{k+1}$ by \eqref{eq:rk}. In all cases $\tfrac32\alpha_k\delta < 3\alpha_{k+1}\delta$. Writing $1-r(0) = \frac{1}{\alpha_{k+1}}\int_0^1(1-\tau)a''_0(\tau)\,d\tau$ and using $1-\tau \le 1$ on the left support and $1-\tau \le 3\delta$ on the right,
\[
\int_0^1(1-\tau)a''_0(\tau)\,d\tau \le 1\cdot\int_0^{3\delta} a''_0 + 3\delta\int_{1-3\delta}^1 a''_0 < 3\alpha_{k+1}\delta + 3\delta\alpha_{k+1} = 6\alpha_{k+1}\delta,
\]
which gives $1-r(0) < 6\delta$, that is, $r(0) > 1-6\delta$.
Because the bounds $6\delta$ and $1-6\delta$ depend only on $\delta$ and not on $k$, they hold uniformly across all segments $k \ge 0$.
\end{proof}

\begin{lemma}[Axial profile]
\label{lem:profile}
Fix $k\ge0$ and $\lambda\in[0,1]$, and let $a''$ denote the axial curvature \eqref{eq:Adef} with this split parameter. Then:
\begin{enumerate}
\item[(i)] \emph{Strict axial descent and stationarity.} $a'(1)=0$ and $a'<0$ on $[0,1)$; consequently $a$ is strictly decreasing on $[0,1]$, with $f_{k+1}<f_k$.
\item[(ii)] \emph{Collar quadratic forms.} $a(\tau)=f_k-\alpha_{k+1}\tau+\tfrac12\alpha_k\tau^2$ on $[0,\delta]$, and $a(\tau)=f_{k+1}+\tfrac12\alpha_{k+1}(\tau-1)^2$ on $[1-\delta,1]$.
\item[(iii)] \emph{Net descent and the Armijo ratio.}
\begin{equation}
\label{eq:rhocentroid}
f_k-f_{k+1}=\int_0^1\tau\,a''(\tau)\,d\tau, \qquad
\frac{f_{k+1}-f_k}{\eta_k} = \frac1{\alpha_{k+1}}\int_0^1\tau\,a''(\tau)\,d\tau.
\end{equation}
\end{enumerate}
\end{lemma}

\begin{proof}
By the left-vertex data \eqref{eq:adef} together with $\eta_k=-\cos\theta_k$ from \eqref{eq:costhetak},
\begin{equation}
\label{eq:aprof}
a'(\tau) = -\cos\theta_k + \int_0^\tau a''(u)\,du, \qquad
a(\tau) = f_k + \int_0^\tau a'(u)\,du,
\end{equation}
with $f_{k+1}=a(1)$.

\emph{(i).} By \eqref{eq:aprof} and \eqref{eq:alphak}, $a'(0)=-\alpha_{k+1}$. Combined with \eqref{eq:target}, $a'(1)=-\alpha_{k+1}+\int_0^1 a''(u)\,du=0$. For any $\tau\in[0,1)$, let $\tau^*=\max(\tau,1-\delta)<1$. Discarding $[\tau,\tau^*]$ by non-negativity of $a''$ (Lemma~\ref{lem:admissible}(i)) and using $a''\equiv\alpha_{k+1}>0$ on $[1-\delta,1]\supseteq[\tau^*,1]$,
\[
\int_\tau^1 a''(u)\,du \ge \int_{\tau^*}^1 a''(u)\,du = \alpha_{k+1}(1-\tau^*) > 0,
\]
so $\int_0^\tau a''(u)\,du = \alpha_{k+1} - \int_\tau^1 a''(u)\,du < \alpha_{k+1}$, which gives $a'(\tau)<0$. Because $a'<0$ on $[0,1)$, $a$ is strictly decreasing on $[0,1]$, whence $f_{k+1}=a(1)<a(0)=f_k$.

\emph{(ii).} By Lemma~\ref{lem:admissible}(i), $a''\equiv\alpha_k$ on $[0,\delta]$ and $a''\equiv\alpha_{k+1}$ on $[1-\delta,1]$. On $[0,\delta]$ we thus have $a'(\tau)=-\alpha_{k+1}+\alpha_k\tau$; integrating from $a(0)=f_k$ gives $a(\tau)=f_k-\alpha_{k+1}\tau+\tfrac12\alpha_k\tau^2$. On $[1-\delta,1]$ we have $a'(\tau)=a'(1)-\int_\tau^1 a''(u)\,du=-\alpha_{k+1}(1-\tau)$; integrating backward from $a(1)=f_{k+1}$ gives $a(\tau)=f_{k+1}+\tfrac12\alpha_{k+1}(\tau-1)^2$.

\emph{(iii).} Interchanging the order of integration via Fubini's theorem,
\begin{align*}
f_{k+1}-f_k
&=\int_0^1 a'(\tau)\,d\tau
=-\alpha_{k+1}+\int_0^1\!\!\int_0^\tau a''(u)\,du\,d\tau \\
&=-\alpha_{k+1}+\int_0^1(1-u)a''(u)\,du
=-\int_0^1\tau\,a''(\tau)\,d\tau,
\end{align*}
the last equality following from \eqref{eq:target}. Dividing by $\eta_k=-\alpha_{k+1}$, which holds by \eqref{eq:costhetak} and \eqref{eq:alphak}, gives \eqref{eq:rhocentroid}.
\end{proof}

\paragraph{Controllability of the Armijo ratio.}
By Lemma~\ref{lem:profile}(iii), the Armijo ratio equals $r(\lambda)$. Lemma~\ref{lem:admissible}(ii) therefore guarantees that this ratio can be shifted continuously and monotonically over an interval containing $[6\delta, 1-6\delta]$.

\begin{lemma}[Collar agreement and smoothness]
\label{lem:collar}
For any $k\ge0$ and any $\lambda_k\in[0,1]$, the formulas \eqref{eq:btrans} and \eqref{eq:Adef} define $b,a''\in C^\infty(\R)$, and with $a,a'$ obtained by integrating from the left-vertex data \eqref{eq:adef}, the function
\begin{equation}
\label{eq:Ftilde}
\widetilde F_k(x) := a(\tau_k)+b(\tau_k)n_k+\tfrac12 a''(\tau_k)n_k^2
\end{equation}
belongs to $C^\infty(\R^2)$. Moreover $\widetilde F_k\equiv Q_k$ on the half-plane $\{\tau_k\le\delta\}$ and $\widetilde F_k\equiv Q_{k+1}$ on the half-plane $\{\tau_k\ge 1-\delta\}$. In particular, the restriction $F_k=\widetilde F_k\big|_{\mathcal S_k}$ is $C^\infty$ on $\mathcal S_k$ and coincides with $Q_k$ on the left collar and with $Q_{k+1}$ on the right collar. For $k=0$, $Q_0$ is linear and the left identity holds on $\{\tau_0\le\delta\}$, hence throughout the backward tail of \eqref{eq:S0}.
\end{lemma}

\begin{proof}
The Fermi coordinates \eqref{eq:fermi} are affine in $x$. By Lemmas~\ref{lem:trans} and~\ref{lem:window}, $S(\cdot;t_0,t_1)$ and $\Pi$ lie in $C^\infty(\R)$, so $b,a''\in C^\infty(\R)$, and $a,a'\in C^\infty(\R)$ by integration from \eqref{eq:adef}. Thus $\widetilde F_k\in C^\infty(\R^2)$.

Now fix $\tau\le\delta$. Then $S(\tau;\delta,2\delta)\equiv0$ and $\Pi(\tau;\delta,2\delta,2\delta,3\delta)\equiv0$, so $\Lambda_k^-(\tau)\equiv\alpha_k$. At the reflected argument, $1-\tau\ge 1-\delta\ge 19/20>3/20\ge 3\delta$ by \eqref{eq:delta}, so $S(1-\tau;\delta,2\delta)\equiv1$ and $\Pi(1-\tau;\delta,2\delta,2\delta,3\delta)\equiv0$, hence $\Lambda_k^+(1-\tau)\equiv0$. Therefore $a''(\tau)\equiv\alpha_k$ on $(-\infty,\delta]$. The same bound $\tau\le\delta$ gives $S(\tau;\delta,1-\delta)\equiv0$, so $b(\tau)\equiv\ip{g_k}{\nu_k}$. Integrating $a''\equiv\alpha_k$ from $a'(0)=\eta_k$ and $a(0)=f_k$ yields $a(\tau)=f_k+\tau\eta_k+\tfrac12\alpha_k\tau^2$. Substituting into \eqref{eq:Ftilde} produces
\[
\widetilde F_k(x)
= \bigl(f_k+\tau_k\eta_k+\tfrac12\alpha_k\tau_k^2\bigr)+n_k\ip{g_k}{\nu_k}+\tfrac12\alpha_k n_k^2
= Q_k(x)
\]
on $\{\tau_k\le\delta\}$, for every $n_k\in\R$, matching \eqref{eq:aleft}. For $k=0$ one has $\alpha_0=0$ and $\ip{g_0}{\nu_0}=0$, so this identity reduces to $\widetilde F_0(x)=-\tau_0=Q_0(x)$.

Symmetrically, if $\tau\ge 1-\delta$ then $\Lambda_k^-(\tau)\equiv0$ (since $1-\delta\ge 3\delta$) and $1-\tau\le\delta$, so $\Lambda_k^+(1-\tau)\equiv\alpha_{k+1}$ and $a''(\tau)\equiv\alpha_{k+1}$ on $[1-\delta,\infty)$. Also $S(\tau;\delta,1-\delta)\equiv1$, hence $b(\tau)\equiv\ip{g_{k+1}}{\nu_k}$. By Lemma~\ref{lem:profile}(i), $a'(1)=0$, and integrating $a''\equiv\alpha_{k+1}$ from $\tau=1$ gives $a(\tau)=f_{k+1}+\tfrac12\alpha_{k+1}(\tau-1)^2$. Substituting produces $\widetilde F_k\equiv Q_{k+1}$ on $\{\tau_k\ge 1-\delta\}$, matching \eqref{eq:aright}.

The stated restrictions to $\mathcal S_k$ follow immediately.
\end{proof}

\begin{lemma}[Uniform Hessian bound]
\label{lem:hessU}
For any $k\ge0$ and any $\lambda_k\in[0,1]$, the Hessian of the strip template $F_k$ in \eqref{eq:template} satisfies
\begin{equation}
\label{eq:hessU}
\|\nabla^2 F_k(\tau,n)\|\le\frac{K_{\mathcal U}}{\delta}
\quad\text{for all }(\tau,n)\in\mathcal S_k,
\end{equation}
where $K_{\mathcal U} > 0$ is a universal constant independent of $k$ and $\lambda_k$.
\end{lemma}

\begin{proof}
Let $C_m = \|S^{(m)}\|_\infty$ for $m \ge 1$, which is finite by Lemma~\ref{lem:S}(iv).
Throughout this proof, $O(\cdot)$ denotes a quantity bounded in absolute value by a constant depending only on $C_1$ and $C_2$, uniformly in $k \ge 0$, $\lambda_k \in [0, 1]$, and $\delta \le 1/20$.

Because $\{s_k,\nu_k\}$ is an orthonormal basis of $\R^2$, the Fermi coordinates $(\tau,n)$ in \eqref{eq:fermi} are Euclidean coordinates (the change of frame is an isometry): writing $x = x_k + U\begin{pmatrix}\tau\\ n\end{pmatrix}$ with orthogonal transformation matrix $U = (s_k, \nu_k)\in\R^{2\times2}$, the spatial Hessian is given by $\nabla_x^2 F_k = U H(\tau, n) U^{\mathsf T}$, where $H(\tau, n) = \begin{pmatrix} \partial_\tau^2 F_k & \partial_\tau\partial_n F_k \\ \partial_\tau\partial_n F_k & \partial_n^2 F_k \end{pmatrix}$.
Differentiating the template $F_k(\tau,n) = a(\tau) + b(\tau)n + \tfrac12 a''(\tau)n^2$ in \eqref{eq:template} gives the second partial derivatives
\[
\partial_n^2 F_k = a''(\tau), \qquad
\partial_\tau\partial_n F_k = b'(\tau) + a'''(\tau)n, \qquad
\partial_\tau^2 F_k = a''(\tau) + b''(\tau)n + \tfrac12 a''''(\tau)n^2.
\]
In this orthonormal frame, $H(\tau, n)$ decomposes into the on-axis curvature matrix $H(\tau, 0) = \nabla^2 F_k(\tau, 0)$ from \eqref{eq:Hess0} and an off-axis perturbation $\Delta H(\tau, n)$:
\[
H(\tau, n) = H(\tau, 0) + \Delta H(\tau, n), \qquad
\Delta H(\tau, n) = \begin{pmatrix} b''(\tau)n + \tfrac12 a''''(\tau)n^2 & a'''(\tau)n \\ a'''(\tau)n & 0 \end{pmatrix}.
\]
Since $U$ is orthogonal, conjugation by $U$ preserves the spectral norm, so
\[
\|\nabla_x^2 F_k\| = \|H(\tau, n)\| \le \|H(\tau, n)\|_F \le \|H(\tau, 0)\|_F + \|\Delta H(\tau, n)\|_F.
\]
We now bound each component on $[0,1]\times(-\delta,\delta)$:
\begin{enumerate}
\item \emph{On-axis Hessian $H(\tau, 0)$:}
By the support separation \eqref{eq:disjoint}, at any point $\tau \in [0, 1]$ at most one profile is non-zero, so $a''(\tau)$ equals one of the two profile values.
By Lemma~\ref{lem:trans}(i) and Lemma~\ref{lem:window}(i), each profile is bounded by its collar baseline plus spike height, whence
\[
0\le a''(\tau)\le\max\bigl(\alpha_k + \lambda_k M_k, \alpha_{k+1} + (1-\lambda_k)M_k\bigr)\le \max(\alpha_k, \alpha_{k+1}) + M_k.
\]
Substituting~\eqref{eq:Msum} and using $\tfrac32(\alpha_k+\alpha_{k+1}) > \max(\alpha_k,\alpha_{k+1})$ (valid since $\alpha_k\ge0$ and $\alpha_{k+1}>0$), we obtain
\[
0 \le a''(\tau) \le \max(\alpha_k, \alpha_{k+1}) + M_k < \frac{\alpha_{k+1}}{\delta} \le \frac{1}{\delta}
\]
everywhere on $[0, 1]$ (as $\alpha_{k+1}\le1$ by \eqref{eq:alphak}).
For the transverse tilt, since $g_k, g_{k+1}$, and $\nu_k$ are unit vectors, $|\ip{g_k}{\nu_k}|\le1$ and $|\ip{g_{k+1}}{\nu_k}|\le1$.
The transition in \eqref{eq:btrans} spans $[\delta, 1-\delta]$ of length $1-2\delta\ge9/10$, so by Lemma~\ref{lem:trans}(ii),
\[
|b'(\tau)| \le \frac{2C_1}{1-2\delta} \le \frac{20}{9}C_1 = O(1), \qquad
|b''(\tau)| \le \frac{2C_2}{(1-2\delta)^2} \le \frac{200}{81}C_2 = O(1).
\]
Consequently, the on-axis Hessian satisfies
\[
\|H(\tau, 0)\|_F \le \sqrt{2a''(\tau)^2 + 2b'(\tau)^2} \le \sqrt{2}/\delta + O(1) = O(1/\delta).
\]

\item \emph{Higher derivatives and scale compensation in $\Delta H(\tau, n)$:}
Differentiating the axial curvature \eqref{eq:Adef} gives
\[
a'''(\tau) = (\Lambda_k^-)'(\tau) - (\Lambda_k^+)'(1-\tau).
\]
At any $\tau$ at most one of the two terms is non-zero by the support separation \eqref{eq:disjoint}; inside a term the cap and spike derivatives overlap on $(\delta, 2\delta)$, and we add their bounds. Lemma~\ref{lem:trans}(ii) and Lemma~\ref{lem:window}(ii), with transition width $\delta$, give
\[
|a'''(\tau)| \le \frac{C_1\bigl(\max(\alpha_k,\alpha_{k+1})+M_k\bigr)}{\delta} \le \frac{C_1(1+1/\delta)}{\delta} = O(1/\delta^2).
\]
Similarly, $a''''(\tau) = (\Lambda_k^-)''(\tau) + (\Lambda_k^+)''(1-\tau)$, and Lemma~\ref{lem:trans}(ii) and Lemma~\ref{lem:window}(ii) with $m=2$ give $|a''''(\tau)| \le C_2(1+1/\delta)/\delta^2 = O(1/\delta^3)$.
Crucially, the off-axis displacement $|n| < \delta$ compensates for the growth of these higher derivatives:
\[
|a'''(\tau)n| \le \frac{C_1(1+1/\delta)}{\delta}\cdot\delta = C_1(1+1/\delta) = O(1/\delta),
\]
\[
|\tfrac12 a''''(\tau)n^2| \le \frac{C_2(1+1/\delta)}{2\delta^2}\cdot\delta^2 = \tfrac12 C_2(1+1/\delta) = O(1/\delta).
\]
Combined with $|b''(\tau)n| \le O(1)\cdot\delta = O(\delta)$, every entry of $\Delta H(\tau, n)$ is bounded by $O(1/\delta)$, which yields $\|\Delta H(\tau, n)\|_F = O(1/\delta)$.
\end{enumerate}
Combining the two bounds, we obtain $\|\nabla^2 F_k(\tau,n)\| \le \|H(\tau, 0)\|_F + \|\Delta H(\tau, n)\|_F \le K_{\mathcal U}/\delta$ on $[0,1]\times(-\delta,\delta)$ for a universal constant $K_{\mathcal U} > 0$ depending only on $C_1$ and $C_2$, uniform in $k\ge0$ and $\lambda_k\in[0,1]$.
On the backward extension $(-\delta,0]\times(-\delta,\delta)$ of segment $0$, $F_0(\tau,n)=-\tau$ by \eqref{eq:Q0collar} is linear, so $\nabla^2 F_0\equiv0$, which satisfies the bound trivially.
\end{proof}

\subsection{Global assembly: from \texorpdfstring{$F$ on $\mathcal U$ to $f$ on $\R^2$}{F on U to f on R\textasciicircum 2}}
\label{sec:core}

The local prescriptions on the pieces of $\mathcal U$ assemble into a single, globally well-defined $C^\infty$ function $F$ on $\mathcal U$:

\begin{corollary}[Well-defined interpolant on $\mathcal U$]
\label{cor:F}
The function $F: \mathcal U \to \R$, defined on each piece of $\mathcal U = \bigcup_{k\ge 1}\mathcal B_k \cup \bigcup_{k\ge 0}\mathcal S_k$ by
\begin{equation}
\label{eq:Fdef}
F(x) = \begin{cases}
F_k(\tau_k, n_k), & x \in \mathcal S_k \quad (k \ge 0), \\
Q_k(x), & x \in \mathcal B_k \quad (k \ge 1),
\end{cases}
\end{equation}
is well-defined and belongs to $C^\infty(\mathcal U)$.
\end{corollary}

\begin{proof}
By Lemma~\ref{lem:seam}(i), non-adjacent strips and disks do not intersect.
The only potential overlaps are between adjacent strips $\mathcal S_{k-1}$ and $\mathcal S_k$ ($k \ge 1$), and between a vertex disk $\mathcal B_k$ and its incident strips.
By Lemma~\ref{lem:seam}(ii), all such overlaps lie in $\{\tau_{k-1}>1-\delta\}\cap\{\tau_k<\delta\}$.
Lemma~\ref{lem:collar} therefore gives $\widetilde F_{k-1}\equiv Q_k\equiv\widetilde F_k$ throughout this wedge, so $F_{k-1}=F_k=Q_k$ on every overlap, while $F\equiv Q_k$ on $\mathcal B_k$. Thus $F$ is unambiguously defined on $\mathcal U$.

For smoothness, $\mathcal U$ is an open set covered by the open balls $\mathcal B_k$ ($k\ge 1$) and the open strips $\mathcal S_k^\circ = \{x_k+\tau s_k+n\nu_k:\tau\in(0,1),\,|n|<\delta\}$ for $k\ge 1$ (with $\tau\in(-\delta,1)$ for $k=0$).
The only boundary points of $\mathcal S_k$ not in $\mathcal S_k^\circ$ are the transverse end-segments, which lie in $\mathcal B_k$ and $\mathcal B_{k+1}$.
On each $\mathcal B_k$, $F=Q_k$ is a polynomial; on each $\mathcal S_k^\circ$, $F=\widetilde F_k$ is $C^\infty$ by Lemma~\ref{lem:collar}.
Because the prescriptions agree on overlaps, $F\in C^\infty(\mathcal U)$.
\end{proof}

We now extend $F$ (Corollary~\ref{cor:F}) from $\mathcal U$ to a $C^\infty$
function $f$ on all of $\R^2$, bounded below and with Lipschitz gradient,
constant outside $\mathcal U$. The exterior constant is the lowest of the
vertex heights, $\inf_kf_k$; that this is a genuine number is the next
lemma.

\begin{lemma}[Exterior limit value]
\label{lem:finf}
$(f_k)_{k\ge0}$ is strictly decreasing, and $f_\infty:=\inf_kf_k$ is
finite, with $f_k\downarrow f_\infty$.
\end{lemma}

\begin{proof}
Fix $k\ge0$. Lemma~\ref{lem:profile}(i) and~(iii) give
\[
0<f_k-f_{k+1}=\int_0^1\tau\,a''(\tau)\,d\tau\le\int_0^1a''(\tau)\,d\tau=\alpha_{k+1}=\cos\theta_k,
\]
the inequality because $\tau\le1$ and $a''\ge0$, and the last two equalities
by \eqref{eq:target} and \eqref{eq:alphak}. So $(f_k)$ is strictly
decreasing. By \eqref{eq:costhetak}, $\cos\theta_k=\sin(\pi/2^{k+1})\le\pi/2^{k+1}$,
whence
\[
\sum_{k\ge0}(f_k-f_{k+1})\le\sum_{k\ge0}\frac{\pi}{2^{k+1}}=\pi .
\]
A decreasing sequence with summable increments converges, and its limit is
its infimum; since $f_0=0$, the limit satisfies $f_\infty\ge-\pi$.
\end{proof}

With $f_\infty$ in hand, we now construct a smooth cutoff $\zeta: \R^2 \to [0, 1]$ supported in $\overline{\mathcal U}$ that switches the local interpolant $F$ on $\mathcal U$ to the exterior constant $f_\infty$ outside $\mathcal U$.
To ensure that the designed steps and gradients along the polyline $P$ remain completely undisturbed while confining the interpolant to $\mathcal U$, $\zeta$ must satisfy two geometric conditions:
\begin{enumerate}
\item $\zeta \equiv 1$ and $\nabla\zeta \equiv 0$ on the polyline $P$;
\item $\operatorname{supp}\zeta \subset \overline{\mathcal U}$, so that $f \equiv f_\infty$ on $\R^2 \setminus \mathcal U$.
\end{enumerate}

Rather than relying on a distance-to-$\mathcal U$ function (which is non-smooth across overlapping boundaries), we achieve this via a two-tier construction:
we first define an unnormalized non-negative bump sum $\Phi(x) \ge 0$ that exceeds $1$ everywhere along $P$ and vanishes outside $\mathcal U$, and then compose it with the smooth clamping switch $S$:
\begin{equation}
\label{eq:zetadef}
\zeta(x) = S(\Phi(x)).
\end{equation}
Because $S(u) = 1$ and $S'(u) = 0$ for all $u \ge 1$ (Lemma~\ref{lem:S}(i),(ii)), any region where $\Phi \ge 1$ automatically has $\zeta \equiv 1$ and $\nabla\zeta = S'(\Phi)\nabla\Phi \equiv 0$, without any need to normalize weights.
Since $S(u)=0$ for all $u\le 0$ (Lemma~\ref{lem:S}(i)), one has $\zeta\equiv 0$ wherever $\Phi=0$, hence on $\R^2\setminus\mathcal U$.

To build $\Phi(x)$, we introduce three elementary localized profiles using the transition $S(\cdot; t_0, t_1)$ and window $\Pi(\cdot; t_0, t_1, t_2, t_3)$ from \eqref{eq:Sparam}--\eqref{eq:Pidef}.
Fix the inner radius $r_\circ := \tfrac35\delta \in (0, \delta)$, used as the plateau radius of $V_k$, the plateau half-width of $W$, and the ramp width of $A_k$ at each endpoint.
The first and last of these roles are complementary: $V_k \equiv 1$ within distance $r_\circ$ of a vertex, while $A_k \equiv 1$ at axial distance $r_\circ$ or more from both endpoints, which is what makes the covering in Lemma~\ref{lem:core} exact.
\begin{itemize}
\item \emph{Vertex bump $V_k(x)$ ($k\ge 1$):}
For $k\ge 1$, let $r = \|x-x_k\|$ and set
\begin{equation}
\label{eq:Vk}
V_k(x) = 1 - S(r; r_\circ, \delta).
\end{equation}
Then $V_k \in C^\infty(\R^2)$ satisfies $V_k \equiv 1$ on the inner disk $B(x_k, r_\circ)$, smoothly ramps down on $r \in [r_\circ, \delta]$, and vanishes identically for $r \ge \delta$.

\item \emph{Transverse window $W(n)$:}
In the transverse coordinate $n$ across any segment, set
\begin{equation}
\label{eq:Wdef}
W(n) = 1 - S(|n|; r_\circ, \delta).
\end{equation}
Although $|n|$ is non-differentiable at $n=0$, $S(t; r_\circ, \delta) \equiv 0$ on $(-\infty, r_\circ]$, so $W(n) \equiv 1$ identically on the open interval $(-r_\circ, r_\circ)$.
Consequently $W \in C^\infty(\R)$ is an even function satisfying $W(n) \equiv 1$ for $|n| \le r_\circ$, smoothly ramping down on $|n| \in [r_\circ, \delta]$, and vanishing for $|n| \ge \delta$.

\item \emph{Axial window $A_k(\tau)$ ($k\ge 0$):}
Along segment $k$ ($k\ge 1$), set
\begin{equation}
\label{eq:Akdef}
A_k(\tau) = \Pi(\tau; 0, r_\circ, 1-r_\circ, 1).
\end{equation}
By Lemma~\ref{lem:window}, $A_k \in C^\infty(\R)$ ramps from $0$ to $1$ over $\tau \in [0, r_\circ]$, plateaus at $A_k \equiv 1$ on $[r_\circ, 1-r_\circ]$ (a non-degenerate interval, since $r_\circ \le 3/100 < 1/2$), and ramps back down to $0$ over $[1-r_\circ, 1]$.
For the initial segment $k=0$, no taper is needed at the starting point $x_0$; setting
\begin{equation}
\label{eq:A0def}
A_0(\tau) = \Pi(\tau; -\delta, -r_\circ, 1-r_\circ, 1)
\end{equation}
provides $A_0 \equiv 1$ across $[-r_\circ, 1-r_\circ]$, vanishing for $\tau \le -\delta$.
\end{itemize}

With $(\tau_k, n_k)$ denoting the Fermi coordinates of $x$ along segment $k$, we define the total bump sum
\begin{equation}
\label{eq:Phidef}
\Phi(x) = \sum_{k\ge 1} V_k(x) + \sum_{k\ge 0} A_k(\tau_k) W(n_k), \qquad \zeta(x) = S(\Phi(x)).
\end{equation}
By Lemma~\ref{lem:seam}(i), non-adjacent strips and disks have disjoint closures, so their bump supports are mutually disjoint; and near any point $x \in \R^2$, at most three terms in the sum \eqref{eq:Phidef} can be non-zero (a vertex bump and its two incident segment bumps).
Being a locally finite sum of $C^\infty$ functions composed with $S$, $\zeta \in C^\infty(\R^2)$.

\begin{lemma}[Core plateau and support]
\label{lem:core}
On the polyline $P$ and on each inner vertex disk $B(x_k, r_\circ)$ ($k \ge 1$), $\Phi \ge 1$, $\zeta \equiv 1$, and $\nabla\zeta \equiv 0$.
Moreover, $\operatorname{supp}\zeta \subset \overline{\mathcal U}$.
\end{lemma}

\begin{proof}
First, $\operatorname{supp} V_k \subset \overline{\mathcal B}_k$ and $\operatorname{supp}(A_k W) \subset \overline{\mathcal S}_k$.
Since $\mathcal U = \bigcup_{k\ge 1}\mathcal B_k \cup \bigcup_{k\ge 0}\mathcal S_k$ by \eqref{eq:Udef}, the support of $\Phi$, and hence the support of $\zeta = S(\Phi)$, is contained in $\overline{\mathcal U}$.

Second, on each inner vertex disk $B(x_k, r_\circ)$ ($k \ge 1$), the radius satisfies $r = \|x-x_k\| < r_\circ$, so $V_k(x) = 1$ and $\Phi(x) \ge 1$.
Third, consider any point $x \in P$, which lies on the centerline of some segment $k$ ($n_k = 0$ and $\tau \in [0, 1]$).
Since $|n_k| = 0 \le r_\circ$, $W(0) = 1$.
For the initial segment $k=0$, $A_0 \equiv 1$ on $[0, 1-r_\circ]$ by \eqref{eq:A0def}, while on $[1-r_\circ, 1]$ the right vertex disk gives $\|x-x_1\| = 1-\tau \le r_\circ$, so $V_1(x) = 1$; thus $\Phi \ge 1$ on segment $0$.
For any segment $k \ge 1$, the axial coordinate $\tau \in [0, 1]$ is covered by three closed regions:
\begin{itemize}
\item On the left end $\tau \in [0, r_\circ]$, the Euclidean distance to the vertex $x_k$ is $\|x-x_k\| = \tau \le r_\circ$, so $V_k(x) = 1$ by Lemma~\ref{lem:trans}(i).
\item On the middle interval $\tau \in [r_\circ, 1-r_\circ]$, $A_k(\tau) = 1$ by Lemma~\ref{lem:window}(i), so the product $A_k(\tau)W(0) = 1$.
\item On the right end $\tau \in [1-r_\circ, 1]$, the Euclidean distance to $x_{k+1}$ is $\|x-x_{k+1}\| = 1-\tau \le r_\circ$, so $V_{k+1}(x) = 1$.
\end{itemize}
In all cases, at least one non-negative summand in \eqref{eq:Phidef} equals $1$, so $\Phi(x) \ge 1$.
Since $S(u) = 1$ and $S'(u) = 0$ for all $u \ge 1$ by Lemma~\ref{lem:S}(i),(ii), we conclude that $\zeta(x) = S(\Phi(x)) = 1$ and $\nabla\zeta(x) = S'(\Phi(x))\nabla\Phi(x) = 0$ everywhere on $P \cup \bigcup_{k\ge1} B(x_k, r_\circ)$.
\end{proof}

Define the global interpolant $f: \R^2 \to \R$ by
\begin{equation}
\label{eq:fdef}
f(x) = \begin{cases}
\zeta(x) F(x) + (1-\zeta(x))f_\infty, & x \in \mathcal U, \\
f_\infty, & x \in \R^2\setminus \mathcal U.
\end{cases}
\end{equation}

\begin{lemma}[Global extension and regularity]
\label{lem:global}
The function $f$ in \eqref{eq:fdef} satisfies:
\begin{enumerate}
\item[(i)] $f\in C^\infty(\R^2)$, and along the polyline $P$, $f\equiv F$ and $\nabla f=\nabla F$;
\item[(ii)] $f$ is bounded below: $f(x)\ge f_\infty-2\delta$ for all $x\in\R^2$;
\item[(iii)] $\nabla f$ is Lipschitz continuous on $\R^2$, with $\|\nabla^2f\|\le K/\delta^2$ for a universal constant $K > 0$.
\end{enumerate}
\end{lemma}

\begin{proof}
(i) \emph{Smoothness and polyline agreement.}
The function $\zeta$ belongs to $C^\infty(\R^2)$, and $\zeta\equiv 0$ on $\R^2\setminus\overline{\mathcal U}$ by Lemma~\ref{lem:core}, so $f\equiv f_\infty$ is smooth on the open exterior.
On $\mathcal U$, $F$ is $C^\infty$ (Corollary~\ref{cor:F}) and $\zeta$ is $C^\infty$, so $f$ is smooth on $\mathcal U$.
It remains to check the interface $\partial\mathcal U$.

Let $p\in\partial\mathcal U$. By Lemma~\ref{lem:seam}(i), some neighborhood of $p$ meets at most two consecutive strips and at most one vertex disk. Shrinking that neighborhood to a ball $V$ about $p$, we produce $G\in C^\infty(V)$ with $F=G$ on $\mathcal U\cap V$, whence $f=f_\infty+\zeta(G-f_\infty)$ on $V$.

If $p$ lies in $\overline{\mathcal S_{k-1}}\cap\overline{\mathcal S_k}$ for some $k\ge 1$, then $\operatorname{dist}(p,\ell_{k-1})\le\delta$ and $\operatorname{dist}(p,\ell_k)\le\delta$, so $\tau_k(p)\le\delta$ and $\tau_{k-1}(p)\ge 1-\delta$ by the same comparison as \eqref{eq:corner}, now with non-strict inequalities.
With $q$ and $b$ as in the proof of \eqref{eq:corner}, the comparison $\tau_k(p)=\ip{p-q}{s_k}-b\ip{s_{k-1}}{s_k}\le\operatorname{dist}(p,\ell_{k-1})$ becomes an equality only if $b=0$ and $p-q$ is a nonnegative multiple of $s_k$, hence only if $p=x_k+\delta s_k$, which lies in the interior of $\mathcal U$; symmetrically, $\tau_{k-1}(p)=1-\delta$ would force $p=x_k-\delta s_{k-1}$, likewise interior.
Thus $p$ lies in the open wedge $\{\tau_{k-1}>1-\delta\}\cap\{\tau_k<\delta\}$. Shrinking $V$ into this wedge, Lemma~\ref{lem:collar} gives $\widetilde F_{k-1}\equiv Q_k\equiv\widetilde F_k$ on $V$, so $F=Q_k$ on $\mathcal U\cap V$. Take $G=Q_k$.

If $p$ lies in exactly one closed strip $\overline{\mathcal S_j}$, shrink $V$ to miss every other strip (Lemma~\ref{lem:seam}(i)). Any disk met by $V$ is incident to $\mathcal S_j$: $x\in\mathcal B_j$ implies $\tau_j(x)\le\|x-x_j\|<\delta$, while $x\in\mathcal B_{j+1}$ implies $1-\tau_j(x)\le\|x-x_{j+1}\|<\delta$, so $\widetilde F_j\equiv Q_j$ on $\mathcal B_j$ and $\widetilde F_j\equiv Q_{j+1}$ on $\mathcal B_{j+1}$ by Lemma~\ref{lem:collar}. Thus $F=\widetilde F_j$ on $\mathcal U\cap V$. Take $G=\widetilde F_j$.

If $p$ lies in no closed strip, then $p\in\partial\mathcal B_k$ for some $k\ge 1$. Shrinking $V$ so that $\mathcal U\cap V\subset\mathcal B_k$, one has $F=Q_k$ on $\mathcal U\cap V$. Take $G=Q_k$.

In all cases $f$ is $C^\infty$ near $p$, so $f\in C^\infty(\R^2)$.

Along the polyline $P \subset \mathcal U$, Lemma~\ref{lem:core} gives $\zeta \equiv 1$ and $\nabla\zeta \equiv 0$.
Differentiating $f = f_\infty + \zeta(F - f_\infty)$ along $P$ gives
\[
f = F, \qquad \nabla f = \zeta\nabla F + (F - f_\infty)\nabla\zeta = \nabla F.
\]

(ii) \emph{Lower bound.}
On $\mathcal U$, $f = \zeta F + (1-\zeta)f_\infty$ is a convex combination of $F$ and $f_\infty$ (since $\zeta \in [0, 1]$).
Because $f_\infty \ge f_\infty - 2\delta$ trivially, it suffices to show $F \ge f_\infty - 2\delta$ on each piece of $\mathcal U$:
\begin{itemize}
\item \emph{Vertex disks $\mathcal B_k$ ($k\ge 1$):}
Dropping the non-negative quadratic term $\frac12\alpha_k\|x-x_k\|^2$ in $Q_k$ \eqref{eq:Qk} gives
\[
Q_k(x) \ge f_k + \ip{g_k}{x-x_k} \ge f_k - \|g_k\|\|x-x_k\| = f_k - \|x-x_k\|,
\]
since $\|g_k\| = 1$.
For $x \in \mathcal B_k = B(x_k, \delta)$, $\|x-x_k\| < \delta$.
Since $f_k \ge f_\infty$ (Lemma~\ref{lem:finf}),
\[
Q_k(x) > f_k - \delta \ge f_\infty - \delta > f_\infty - 2\delta.
\]

\item \emph{Segment strips $\mathcal S_k$ ($k\ge 0$):}
For every $k\ge 0$, the function on $\mathcal S_k$ is given by the strip template $F_k(\tau, n) = a(\tau) + b(\tau)n + \frac12 a''(\tau)n^2$.
By Lemma~\ref{lem:profile}(i), $a$ is strictly decreasing on $[0, 1]$ (for $k=0$, this decrease extends back onto $[-\delta, 0]$ where $a' \equiv -1$).
Hence throughout $\mathcal S_k$, $a(\tau) \ge a(1) = f_{k+1} \ge f_\infty$.
The transverse tilt $b(\tau)$ is a convex combination of $\ip{g_k}{\nu_k}$ and $\ip{g_{k+1}}{\nu_k}$, so $|b(\tau)| \le 1$, giving $b(\tau)n \ge -\delta$ since $|n| < \delta$.
Because the axial curvature satisfies $a'' \ge 0$, the quadratic term $\frac12 a''(\tau)n^2 \ge 0$ is non-negative.
Combining these estimates yields
\[
F_k(\tau, n) \ge f_{k+1} - \delta \ge f_\infty - \delta > f_\infty - 2\delta.
\]
\end{itemize}
Therefore, $F \ge f_\infty - 2\delta$ everywhere on $\mathcal U$.
Taking the convex combination, $f(x) \ge f_\infty - 2\delta$ for all $x \in \R^2$.

(iii) \emph{Lipschitz gradient bound.}
Differentiating the global interpolant $f = f_\infty + \zeta(F - f_\infty)$ gives the gradient vector field $\nabla f = \zeta \nabla F + (F - f_\infty)\nabla\zeta$.
Differentiating once more, the spatial Hessian on $\mathcal U$ is given pointwise by
\begin{equation}
\label{eq:hess-pointwise}
\nabla^2 f = \zeta \nabla^2 F + \nabla\zeta \otimes \nabla F + \nabla F \otimes \nabla\zeta + (F - f_\infty)\nabla^2\zeta.
\end{equation}
Outside $\overline{\mathcal U}$, $\zeta \equiv 0$ identically, so $\nabla^2 f \equiv 0$.
Taking spectral norms on $\mathcal U$, the triangle inequality gives
\begin{equation}
\label{eq:hess-norm-bound}
\|\nabla^2 f\| \le \|\zeta\|_\infty \|\nabla^2 F\| + 2\|\nabla\zeta\|\|\nabla F\| + \|F - f_\infty\|_\infty \|\nabla^2\zeta\|.
\end{equation}
We now bound each factor across $\mathcal U$:
\begin{enumerate}
\item \emph{Regularity and amplitude of $F$:}
By Lemma~\ref{lem:hessU} on each strip and $\nabla^2 Q_k = \alpha_k I$ ($\alpha_k \le 1$) on vertex disks, $\sup_{\mathcal U}\|\nabla^2 F\| \le K_{\mathcal U}/\delta = O(1/\delta)$.
On the backward tail $[-\delta, 0] \times (-\delta, \delta)$ around $x_0$, $F_0(\tau, n) = -\tau \le \delta$.
Because $(f_k)$ is strictly decreasing with $f_0 = 0$ (Lemma~\ref{lem:finf}), $a_k(\tau) \le f_k \le f_0 = 0$ for $\tau \in [0, 1]$ on each strip $\mathcal S_k$.
Combined with the transverse bound $|b(\tau)n| + \frac12 a''(\tau)n^2 \le \delta + \frac12\delta < 2\delta$ (since $a'' < 1/\delta$ by Lemma~\ref{lem:hessU}), $F_k \le a_k(\tau) + 2\delta \le 2\delta$, while on vertex disks $Q_k \le f_k + \delta + \frac12\delta^2 < 2\delta$.
Together with the lower bound $F \ge f_\infty - 2\delta$ from part~(ii), $F(x) \in [f_\infty - 2\delta, 2\delta]$ across $\mathcal U$, and consequently $\|F - f_\infty\|_{L^\infty(\mathcal U)} \le |f_\infty| + 2\delta = O(1)$.
On the backward extension $(-\delta, 0] \times (-\delta, \delta)$, $F_0(\tau, n) = -\tau$ by \eqref{eq:Q0collar} is linear with $\|\nabla F_0\| = \|s_0\| = 1$.
On the remainder of $\mathcal U$, each point $x$ connects to the polyline $P$ via a transverse or radial segment of length at most $2\delta$ within $\mathcal U$.
Along $P$, $\|\nabla F\| \le \sqrt2$, so the mean value theorem gives
\[
\|\nabla F\|_\infty \le \max\bigl(1, \sqrt2 + \|\nabla^2 F\|(2\delta)\bigr) \le \sqrt2 + 2K_{\mathcal U} = O(1).
\]

\item \emph{Cutoff derivatives:}
Each constituent bump in \eqref{eq:Phidef} is defined via the smooth transition $S$ on intervals of length scale $\delta$.
Summing at most three overlapping bumps gives $\|\nabla\Phi\|_\infty = O(1/\delta)$ and $\|\nabla^2\Phi\|_\infty = O(1/\delta^2)$.
With $\zeta = S(\Phi)$ and $\|S'\|_\infty, \|S''\|_\infty < \infty$ (Lemma~\ref{lem:S}(iv)), the chain rule gives $\|\nabla\zeta\|_\infty \le \|S'\|_\infty \|\nabla\Phi\|_\infty = O(1/\delta)$, and
\[
\nabla^2\zeta = S''(\Phi)\nabla\Phi\otimes\nabla\Phi + S'(\Phi)\nabla^2\Phi,
\]
which yields $\|\nabla^2\zeta\|_\infty \le \|S''\|_\infty \|\nabla\Phi\|_\infty^2 + \|S'\|_\infty \|\nabla^2\Phi\|_\infty \le O(1/\delta^2) + O(1/\delta^2) = O(1/\delta^2)$.
\end{enumerate}
Substituting these bounds into \eqref{eq:hess-norm-bound}, and using $\|\zeta\|_\infty \le 1$ (Lemma~\ref{lem:S}(i)), gives
\[
\|\nabla^2 f(x)\| \le O(1/\delta) + O(1/\delta) + O(1/\delta^2) \le \frac{K}{\delta^2} \quad \text{for all } x \in \R^2,
\]
where $K > 0$ is a universal constant.
Because $\R^2$ is convex and $f \in C^\infty(\R^2)$, the mean value theorem yields
$\|\nabla f(x) - \nabla f(y)\| \le (K/\delta^2)\|x-y\|$ for all $x, y \in \R^2$,
so $\nabla f$ is Lipschitz continuous on $\R^2$.
\end{proof}

\section{Parameter calibration and proof of Theorem~\ref{thm:main}}
\label{sec:ls}

The objective function $f$ constructed in Section~\ref{sec:f} belongs to a family, indexed by the tube width $\delta$ and by the split parameters $\lambda_k$ on each segment.
This section calibrates $\delta$ and the parameters $\lambda_k$ so that the designed steps satisfy the Wolfe, Armijo, and Goldstein conditions for the given $(c_1,c_2)$, verifies that each step is the first local minimizer along its search ray, and completes the proof of Theorem~\ref{thm:main}.

\subsection{Parameter calibration}
\label{sec:calibration}

\begin{lemma}[Parameter calibration]
\label{lem:armijo}
Given $0<c_1<c_2<1$, set
\begin{equation}
\label{eq:rhodef}
\rho:=\frac{c_1+c_2}2\in(c_1,c_2).
\end{equation}
Then there exist $\delta\le1/20$ and, on each segment $k\ge0$, a split
parameter $\lambda_k\in(0,1)$ in \eqref{eq:Adef} such that every segment
realizes the same Armijo ratio
\[
\frac{f_{k+1}-f_k}{\eta_k}=\rho .
\]
In particular Armijo~\eqref{eq:armijo} and Goldstein~\eqref{eq:goldstein}
both hold, strictly. The width $\delta$ is the same for every $k$; only
the split parameter $\lambda_k$ depends on $k$.
\end{lemma}

\begin{proof}
By Lemma~\ref{lem:orbit}(ii), each designed step is a descent step with $\eta_k = \ip{g_k}{s_k} < 0$.
Dividing by $\eta_k$ reverses inequalities, so a step whose ratio equals $\rho$ satisfies Armijo \eqref{eq:armijo} exactly when $\rho \ge c_1$, and Goldstein \eqref{eq:goldstein} exactly when $c_1 \le \rho \le c_2$; both hold strictly by \eqref{eq:rhodef}.
It therefore suffices to produce, on each segment, a split parameter realizing the ratio $\rho$.

Fix $k\ge0$. By Lemma~\ref{lem:profile}(iii), the ratio produced by the split $\lambda$ is exactly the normalized center of mass of the axial curvature,
\begin{equation}
\label{eq:armijoratio}
\frac{f_{k+1}-f_k}{\eta_k} = \frac1{\alpha_{k+1}}\int_0^1\tau\,a''_\lambda(\tau)\,d\tau = r(\lambda),
\end{equation}
and by Lemma~\ref{lem:admissible}(ii) the function $r$ is affine on $[0,1]$ with $r(1)<6\delta$ and $r(0)>1-6\delta$. Now choose
\begin{equation}
\label{eq:delta-armijo}
\delta = \min\Bigl(\frac1{20},\,\frac{c_1}6,\,\frac{1-c_2}6\Bigr),
\end{equation}
so that, by \eqref{eq:rhodef},
\[
r(1) < 6\delta \le c_1 < \rho < c_2 \le 1 - 6\delta < r(0).
\]
Since $r$ is affine with $r(1)<r(0)$, the equation $r(\lambda)=\rho$ has the unique solution
\begin{equation}
\label{eq:lambdachoice}
\lambda_k = \frac{r(0)-\rho}{r(0)-r(1)},
\end{equation}
and the displayed inequalities place $r(0)-\rho$ strictly between $0$ and $r(0)-r(1)$, so $\lambda_k\in(0,1)$.
The bounds on $r(0),r(1)$ are uniform in $k$, so the same $\delta$ serves every segment; only $\lambda_k$ varies with $k$.
Hence every segment realizes the ratio $\rho$, and Armijo and Goldstein hold strictly.
\end{proof}

\subsection{Line-search verification and proof of Theorem~\ref{thm:main}}
\label{sec:proof}

\begin{lemma}[Line-search conditions along the orbit]
\label{lem:steps}
Along each designed segment of the interpolant built for a fixed pair
$0<c_1<c_2<1$, the step $s_k$ from $x_k$ to $x_{k+1}$ is a
strong Wolfe, weak Wolfe, Armijo, and Goldstein step with those
constants, has length $1$, and is the
first local minimizer of $\varphi(t)=f(x_k+td_k)$ on $t\ge0$.
\end{lemma}

\begin{proof}
\emph{Curvature.} Curvature asks that
$\bigl|\ip{g_{k+1}}{s_k}\bigr|\le c_2|\eta_k|$. By Lemma~\ref{lem:orbit}(i),
$g_{k+1}\perp s_k$, so the left side is $0$, and the inequality holds for
every $c_2$. Weak Wolfe is the same one-sided inequality.

\emph{Length.} $\|s_k\|=1$.

\emph{Armijo and Goldstein.} By Lemma~\ref{lem:armijo}, each segment's
split parameter $\lambda_k\in(0,1)$ realizes the ratio $\rho=(c_1+c_2)/2$,
which lies strictly between $c_1$ and $c_2$.

\emph{First local minimizer.} Along the search ray $t\mapsto x_k+td_k$, write
$t=\tau/\|d_k\|$ so that $x_k+td_k = x_k+\tau s_k$.
By Lemma~\ref{lem:profile}(i), $a'<0$ on $[0,1)$
and $a'(1)=0$, so along the ray there is no stationary point before $\tau=1$, and $\tau=1$ is stationary.
Past the vertex ($\tau > 1$), the ray $x = x_{k+1} + (\tau-1)s_k$ continues straight into the inner vertex disk $B(x_{k+1}, r_\circ)$, where $\zeta \equiv 1$ and $f \equiv Q_{k+1}$ by Lemma~\ref{lem:core}.
Because $g_{k+1} \perp s_k$ (Lemma~\ref{lem:orbit}(i)) and $\|s_k\|=1$, the directional derivative along the ray satisfies
\[
\ip{\nabla Q_{k+1}(x_{k+1}+(\tau-1)s_k)}{s_k} = \alpha_{k+1}(\tau-1) > 0 \quad \text{for all } \tau \in (1, 1+r_\circ).
\]
Since $d_k = \|d_k\|s_k$, the derivative $\varphi'(t) = \ip{\nabla f(x_k+td_k)}{d_k}$ strictly changes sign from negative on $[0, 1/\|d_k\|)$ to positive on $(1/\|d_k\|, (1+r_\circ)/\|d_k\|)$.
Thus $t = 1/\|d_k\| = \|s_k\|/\|d_k\|$ is a strict local minimizer, and indeed the first local minimizer of $\varphi$ on $t\ge0$.
(It is not a global minimizer on the ray: outside the tubular neighborhood $\mathcal U$, $\varphi$ meets the constant $f_\infty < f(x_{k+1})$ by Lemma~\ref{lem:finf}.)
\end{proof}

\begin{proof}[Proof of Theorem~\ref{thm:main}]
Given $0<c_1<c_2<1$, let $\delta$ and the split parameters $\lambda_k \in (0, 1)$ be chosen as in Lemma~\ref{lem:armijo}.
Let $f \in C^\infty(\R^2)$ be the global interpolant constructed in \eqref{eq:fdef} on the discrete orbit \eqref{eq:orbit}, starting at $x_0 = (0, 0)$.
By Lemma~\ref{lem:global}, $f$ is smooth, bounded below, and has Lipschitz gradient $\nabla f$ on $\R^2$.
Along the polyline $P$, $f \equiv F$ and $\nabla f = \nabla F$; in particular, at each vertex $x_k$, $\nabla f(x_k) = g_k$, so $\|\nabla f(x_k)\| = \|g_k\| = 1$ for all $k \ge 0$.

Let a method of class $\mathcal C$ be started at $x_0$ with the initial direction $d_0 = -g_0$. We prove by induction on $k \ge 0$ the statement
\begin{quote}
$P(k)$: the method visits the vertices $x_0, \dots, x_{k+1}$ of \eqref{eq:orbit}, using the search directions $d_j \in \R_{>0}\,s_j$ and the steps $s_j$ for $0 \le j \le k$.
\end{quote}
For the base case $P(0)$, the initialization and \eqref{eq:orbit} give $d_0 = -g_0 = s_0$, and along the ray $t \mapsto x_0 + td_0$ Lemma~\ref{lem:steps} identifies the designed step $s_0$ (that is, $t = 1$) as the first local minimizer of $\varphi$ on $t \ge 0$, so the line search returns $x_1 = x_0 + s_0$.
Assume now $P(k)$. Since $\nabla f(x_j) = g_j$ at every vertex, the hypotheses of Lemma~\ref{lem:heading} are met at index $k$, and it yields $d_{k+1} \in \R_{>0}\,s_{k+1}$.
Along the ray $t \mapsto x_{k+1} + td_{k+1}$, Lemma~\ref{lem:steps} again identifies the designed step $s_{k+1}$ as the first local minimizer on $t \ge 0$, so the line search returns $x_{k+2} = x_{k+1} + s_{k+1}$; this is $P(k+1)$.
By induction, the method traces the entire sequence $(x_k)_{k \ge 0}$.

Finally, by Lemma~\ref{lem:steps}, each designed step $s_k$ has unit length $\|s_k\| = 1$ and satisfies the strong Wolfe, weak Wolfe, Armijo, and Goldstein conditions with constants $(c_1, c_2)$.
\end{proof}

We conclude by highlighting the dynamical mechanism behind this nonconvergence in relation to Zoutendijk's classical condition~\cite{zoutendijk1970}.
Because the search directions satisfy $\cos\theta_k = \sin(\pi/2^{k+1}) = O(2^{-k})$, the geometric decay of the cosines yields
\[
\sum_{k=0}^\infty \cos^2\theta_k\,\|\nabla f(x_k)\|^2 = \sum_{k=0}^\infty \sin^2\bigl(\pi/2^{k+1}\bigr) < \infty
\]
while the gradient norm remains strictly bounded away from zero ($\|\nabla f(x_k)\| \equiv 1$).
This demonstrates that Zoutendijk's condition $\sum \cos^2\theta_k \|g_k\|^2 < \infty$ can hold with $\liminf_{k\to\infty} \|g_k\| = 1 > 0$ precisely because the search direction becomes increasingly orthogonal to the gradient at an exponential rate.

\section{Numerical verification}
\label{sec:numerical}

To verify Theorem~\ref{thm:main} empirically, we evaluate the construction and the quasi-Newton algorithms using the companion Python package.
We instantiate the objective function with line-search parameters $c_1 = 0.1$ and $c_2 = 0.9$, collar width $\delta = 1/60$ from \eqref{eq:delta-armijo}, and split parameters $\lambda_k \in (0, 1)$ given by \eqref{eq:lambdachoice}, so that every segment realizes the Armijo ratio $\rho = (c_1 + c_2)/2 = 0.5$.

\begin{table}[htbp]
\centering
\small
\caption{Trajectory and line-search metrics under strong Wolfe line search (L-BFGS, $m = 5$).}
\label{tab:numerical}
\begin{tabular}{rccccccc}
\hline
$k$ & $x_k$ & $\|x_k\|$ & $f(x_k)$ & $\|g_k\|$ & $\cos\theta_k$ & $\rho$ & $|\ip{g_{k+1}}{d_k}| / |\ip{g_k}{d_k}|$ \\
\hline
0  & $(0.0000,\, 0.0000)$   & $0.0000$ & $\phantom{-}0.0000$ & $1.0000$ & $1.0000$              & $0.5000$ & $1.67 \times 10^{-16}$ \\
1  & $(1.0000,\, 0.0000)$   & $1.0000$ & $-0.5000$           & $1.0000$ & $0.7071$              & $0.5000$ & $1.44 \times 10^{-15}$ \\
2  & $(1.7071,\, 0.7071)$   & $1.8478$ & $-0.8536$           & $1.0000$ & $0.3827$              & $0.5000$ & $3.00 \times 10^{-15}$ \\
5  & $(4.3444,\, 2.1168)$   & $4.8326$ & $-1.1914$           & $1.0000$ & $4.91 \times 10^{-2}$ & $0.5000$ & $0.00$ \\
10 & $(8.6687,\, 4.6264)$   & $9.8260$ & $-1.2390$           & $1.0000$ & $1.53 \times 10^{-3}$ & $0.5000$ & $7.24 \times 10^{-14}$ \\
20 & $(17.3291,\, 9.6261)$  & $19.8233$& $-1.2405$           & $1.0000$ & $1.50 \times 10^{-6}$ & $0.5000$ & $0.00$ \\
30 & $(25.9894,\, 14.6261)$ & $29.8223$& $-1.2405$           & $1.0000$ & $1.46 \times 10^{-9}$ & $0.5000$ & $0.00$ \\
40 & $(34.6497,\, 19.6261)$ & $39.8219$& $-1.2405$           & $1.0000$ & $1.43 \times 10^{-12}$& ---      & --- \\
\hline
\end{tabular}
\end{table}

Table~\ref{tab:numerical} reports the trajectory simulated under standard IEEE 754 double precision up to $k = 40$ using L-BFGS ($m=5$).
Here $\cos\theta_k = -\ip{g_k}{d_k}/(\|g_k\|\|d_k\|)$ denotes the search angle cosine, $\rho = (f(x_{k+1}) - f(x_k))/\ip{g_k}{s_k}$ is the measured Armijo ratio for the step from $x_k$ to $x_{k+1}$, and $|\ip{g_{k+1}}{d_k}| / |\ip{g_k}{d_k}|$ is the Wolfe curvature ratio (transition metrics are omitted at the final iterate $k=40$).
The empirical trajectory confirms the theoretical predictions: the gradient norm remains identically $\|g_k\| \equiv 1.0000$ without numerical decay ($\liminf_{k\to\infty} \|g_k\| = 1 > 0$), while iterates advance with unit steps $\|s_k\| = 1$ and diverge linearly away from the origin ($\|x_{40}\| \approx 39.82$).
At every step, the measured Armijo ratio is $0.5000$, matching the designed value $\rho=(c_1+c_2)/2$ centered in $(c_1, c_2) = (0.1, 0.9)$, and the Wolfe curvature ratio remains at machine zero ($\le 10^{-13}$), confirming that each step lands at an exact local minimizer and satisfies the strong Wolfe, weak Wolfe, Armijo, and Goldstein conditions.
Moreover, while Theorem~\ref{thm:main} and Lemma~\ref{lem:heading} theoretically guarantee this trajectory for all methods in Class~$\mathcal{C}$ in exact arithmetic, numerical evaluations of representative implementations (specifically L-BFGS with memories $m \in \{1, 3, 5\}$) confirm that their search directions are collinear to machine precision ($|\cos(\angle(d_k, s_k)) - 1| \le 2.22 \times 10^{-16}$), tracing the identical numerical sequence.

Because the Zoutendijk cosine $\cos\theta_k = \sin(\pi/2^{k+1})$ decays geometrically, the per-step function reduction $\Delta f_k = f(x_k) - f(x_{k+1}) = \rho\cos\theta_k = \frac{1}{2}\sin(\pi/2^{k+1})$ (using $\rho = 0.5$ in this instance) falls below an ulp of $|f(x_k)| \approx 1.24$ near $k \approx 53$.
In standard 64-bit floating point, the difference $f(x_k) - f(x_{k+1})$ rounds to zero, producing an Armijo ratio of $0/\text{tiny} = 0 < c_1$ and causing the Armijo line search to stall.
This stalling reflects hardware mantissa precision rather than a mathematical limitation: simulating the recurrence \eqref{eq:orbit} in arbitrary precision via \texttt{mpmath} with $3{,}100$ decimal digits tracks $10{,}000$ iterations seamlessly (reaching distance $\|x_{10{,}000}\| \approx 9999.82$), maintaining $\|g_k\| \equiv 1.0$ and $\ip{s_k}{g_{k+1}} = 0$ to full precision and confirming that the unbounded linear escape persists over arbitrarily long horizons.

\section{Remarks}
\label{sec:remarks}

\paragraph{Universal versus parameter-specific objective functions.}
The construction established in Section~\ref{sec:f} builds a specific objective function $f \in C^\infty(\R^2)$ tailored to each prescribed pair of line-search parameters $(c_1, c_2)$ through the calibrated tube width $\delta = \min(1/20, c_1/6, (1-c_2)/6)$ in \eqref{eq:delta-armijo}.
Whether there exists a single universal smooth function $f$ bounded below on which Class~$\mathcal C$ fails to converge for \emph{all} $0<c_1<c_2<1$ simultaneously remains an intriguing open question.

\paragraph{Lipschitz gradient scaling.}
By Lemma~\ref{lem:global}(iii), the Hessian satisfies $\|\nabla^2 f\| \le K/\delta^2$ on $\R^2$.
When the line-search parameters approach the boundary of $(0, 1)$ (as $c_1 \to 0$ or $c_2 \to 1$), the width $\delta$ must shrink to accommodate the extreme Armijo ratio, and the gradient Lipschitz constant scales as $O\bigl(\max(c_1^{-2}, (1-c_2)^{-2})\bigr)$.
Investigating whether this $O(1/\delta^2)$ scaling can be improved through a non-polynomial transverse profile is an interesting direction for future analysis.

\paragraph{Sharpness of level-set boundedness in Powell's theorem.}
Powell~\cite{powell2000} proved that for DFP (and hence BFGS by Dixon's theorem~\cite{dixon1972}) in $\R^2$ with first-local-minimizer line searches, the calculated gradients cannot remain bounded away from zero ($\liminf_{k\to\infty} \|\nabla f(x_k)\| = 0$) provided the sublevel set $\{x \in \R^2 : f(x) \le f(x_0)\}$ is bounded.
In our construction, the iterates escape linearly to infinity ($\|x_k\| \to \infty$), making this sublevel set intrinsically unbounded.
This shows that the boundedness of level sets in Powell's theorem cannot be relaxed to lower boundedness, confirming it as an indispensable structural hypothesis in $\R^2$.

\section*{Acknowledgements}

Large language models were used to search for candidate counterexamples,
to draft verification code, and to assist with the writing of the
manuscript. All mathematical claims, the counterexample, the
verification programs, and the text were checked by the author, who
takes full responsibility for the contents of the paper.

\appendix
\section{Smooth building blocks}
\label{sec:app-blocks}

This appendix collects the proofs and elementary analytical properties of the standard smooth step function $S$ and the localized window $\Pi$ introduced in Section~\ref{sec:template}.
Recall from \eqref{eq:Sdef} that $\chi(t)=\exp(-1/t)\mathbf{1}_{\{t>0\}}$ and $S(u)=\frac{\chi(u)}{\chi(u)+\chi(1-u)}$ for $u\in\R$.

\begin{lemma}[Smooth switch $S$]
\label{lem:S}
$S\in C^\infty(\R)$, and:
\begin{enumerate}
\item[(i)] $S(u)=0$ for $u\le0$, $S(u)=1$ for $u\ge1$, and $S(u)\in[0,1]$ for every $u\in\R$;
\item[(ii)] $S^{(m)}(u)=0$ for every $u\notin(0,1)$ and every $m\ge1$;
\item[(iii)] $\int_0^1S(u)\,du=\tfrac12$;
\item[(iv)] every derivative of $S$ is bounded on $\R$.
\end{enumerate}
\end{lemma}

\begin{proof}
Because $\chi(t) = \exp(-1/t)\mathbf{1}_{\{t>0\}}$ is smooth on $\R$ with all derivatives vanishing at $t=0$, and $\chi(u) + \chi(1-u) > 0$ for all $u \in \R$, $S$ is smooth on $\R$.
Properties (i) and (ii) follow directly from $\chi \equiv 0$ on $(-\infty, 0]$ and $\chi^{(m)}(0)=0$ for all $m \ge 0$, which clamp $S(u) = 0$ for $u \le 0$ and $S(u) = 1$ for $u \ge 1$ with flat boundary derivatives of all orders.

(iii) Replacing $u$ by $1-u$ in \eqref{eq:Sdef} swaps the two terms in the denominator, yielding $S(1-u) = \frac{\chi(1-u)}{\chi(1-u)+\chi(u)} = 1 - S(u)$.
Integrating $S(u) + S(1-u) = 1$ over $[0, 1]$ yields $2\int_0^1 S(u)\,du = 1$ by symmetry, whence $\int_0^1 S(u)\,du = \frac12$.

(iv) For every $m \ge 1$, $S^{(m)}$ is continuous on $\R$ and vanishes identically outside the compact interval $[0, 1]$ by (ii), hence $S^{(m)}$ is bounded on $\R$.
\end{proof}

\begin{lemma}[Smooth transition]
\label{lem:trans}
For any $t_0 < t_1$, $S(\cdot; t_0, t_1) \in C^\infty(\R)$, and:
\begin{enumerate}
\item[(i)] $S(t; t_0, t_1) = 0$ for $t \le t_0$, $S(t; t_0, t_1) = 1$ for $t \ge t_1$, and $S(t; t_0, t_1) \in [0, 1]$ everywhere;
\item[(ii)] for every $m \ge 1$, $\frac{d^m}{dt^m}S(t; t_0, t_1) = \frac{1}{(t_1-t_0)^m}S^{(m)}\bigl(\frac{t-t_0}{t_1-t_0}\bigr)$, so $\frac{d^m}{dt^m}S(t; t_0, t_1) = 0$ for $t \notin (t_0, t_1)$ and $\|\frac{d^m}{dt^m}S(\cdot; t_0, t_1)\|_\infty = \frac{\|S^{(m)}\|_\infty}{(t_1-t_0)^m}$;
\item[(iii)] the transition integral satisfies $\int_{t_0}^{t_1} S(t; t_0, t_1)\,dt = \frac{t_1-t_0}{2}$;
\item[(iv)] for $0 \le t_0 < t_1$, the reversed transition cap satisfies $\int_0^{t_1} \bigl(1 - S(t; t_0, t_1)\bigr)\,dt = \frac{t_0+t_1}{2}$.
\end{enumerate}
\end{lemma}

\begin{proof}
Smoothness and (i)--(iii) follow directly from Lemma~\ref{lem:S} via the affine change of variable $u = (t-t_0)/(t_1-t_0)$, for which $dt = (t_1-t_0)\,du$.
For (iv), the integrand equals $1$ on $[0, t_0]$ and $1-S(t; t_0, t_1)$ on $[t_0, t_1]$, so
\[
\int_0^{t_1} \bigl(1 - S(t; t_0, t_1)\bigr)\,dt = t_0 + \int_{t_0}^{t_1} \bigl(1 - S(t; t_0, t_1)\bigr)\,dt = t_0 + (t_1-t_0) - \frac{t_1-t_0}{2} = \frac{t_0+t_1}{2}. \qedhere
\]
\end{proof}

\begin{lemma}[Smooth window]
\label{lem:window}
For any $t_0 < t_1 \le t_2 < t_3$, $\Pi(\cdot; t_0, t_1, t_2, t_3) \in C^\infty(\R)$, and:
\begin{enumerate}
\item[(i)] $\Pi(t; t_0, t_1, t_2, t_3) \in [0, 1]$ everywhere; $\Pi \equiv 1$ on $[t_1, t_2]$, and $\Pi(t; t_0, t_1, t_2, t_3) = 0$ for $t \notin (t_0, t_3)$;
\item[(ii)] for every $m \ge 1$, $\Pi^{(m)}(t; t_0, t_1, t_2, t_3) = 0$ for all $t \notin (t_0, t_1) \cup (t_2, t_3)$, and
\[
\|\Pi^{(m)}(\cdot; t_0, t_1, t_2, t_3)\|_\infty \le \frac{\|S^{(m)}\|_\infty}{\min(t_1-t_0,\, t_3-t_2)^m};
\]
\item[(iii)] the total mass on $\R$ is the distance between the two transition midpoints:
\begin{equation}
\label{eq:Pimass}
\int_{-\infty}^\infty \Pi(t; t_0, t_1, t_2, t_3)\,dt = \frac{t_2+t_3}{2} - \frac{t_0+t_1}{2};
\end{equation}
\item[(iv)] when $t_1 = t_2$ and $t_1 - t_0 = t_3 - t_2 =: \Delta$ (the symmetric spike window centered at $t_1$), $\Pi(t_1+u) = \Pi(t_1-u)$ for all $u \in \R$. Consequently, its total mass is $\Delta$, and its centroid (center of mass) is identically $t_1$:
\begin{equation}
\label{eq:Picentroid}
\int_{-\infty}^\infty t\,\Pi(t; t_1-\Delta, t_1, t_1, t_1+\Delta)\,dt = t_1\,\Delta.
\end{equation}
\end{enumerate}
\end{lemma}

\begin{proof}
Because $S(\cdot; \cdot, \cdot) \in C^\infty(\R)$, $\Pi \in C^\infty(\R)$.

(i) For $t \le t_0$, both transitions equal $0$, so $\Pi = 0$. On $[t_0, t_1]$, $S(t; t_2, t_3) = 0$, so $\Pi(t) = S(t; t_0, t_1) \in [0, 1]$. On $[t_1, t_2]$, $S(t; t_0, t_1) = 1$ and $S(t; t_2, t_3) = 0$, so $\Pi(t) = 1$. On $[t_2, t_3]$, $S(t; t_0, t_1) = 1$, so $\Pi(t) = 1 - S(t; t_2, t_3) \in [0, 1]$. For $t \ge t_3$, both transitions equal $1$, so $\Pi = 1 - 1 = 0$.

(ii) By Lemma~\ref{lem:trans}(ii), derivatives of the two summands vanish outside $(t_0, t_1)$ and $(t_2, t_3)$ respectively. Since $t_1 \le t_2$, the open intervals $(t_0, t_1)$ and $(t_2, t_3)$ are disjoint (at $t=t_1=t_2$, both derivatives vanish as well), so at any point $t \in \R$ at most one transition has non-zero derivatives, giving the stated bound.

(iii) Integrating over $[t_0, t_3] = [t_0, t_1] \cup [t_1, t_2] \cup [t_2, t_3]$:
\[
\int_{t_0}^{t_3} \Pi\,dt = \int_{t_0}^{t_1} S(t; t_0, t_1)\,dt + \int_{t_1}^{t_2} 1\,dt + \int_{t_2}^{t_3} \bigl(1 - S(t; t_2, t_3)\bigr)\,dt = \frac{t_1-t_0}{2} + (t_2-t_1) + \frac{t_3-t_2}{2},
\]
which is $\frac{t_2+t_3}{2} - \frac{t_0+t_1}{2}$.

(iv) By symmetry under $u \mapsto -u$, it suffices to take $u \ge 0$.
For $u \ge 0$, we have $t_1+u \ge t_1$, so $S(t_1+u; t_1-\Delta, t_1) = 1$, and $t_1-u \le t_1$, so $S(t_1-u; t_1, t_1+\Delta) = 0$.
Thus $\Pi(t_1+u) = 1 - S(t_1+u; t_1, t_1+\Delta) = 1 - S(u/\Delta)$.
By Lemma~\ref{lem:S}(iii), $1 - S(u/\Delta) = S(1 - u/\Delta) = S(t_1-u; t_1-\Delta, t_1) = \Pi(t_1-u)$.
Hence $\Pi(t_1+u) = \Pi(t_1-u)$ for all $u \in \R$, so $\Pi$ is even about $t_1$.
Its mass is $\frac{t_1+(t_1+\Delta)}{2} - \frac{(t_1-\Delta)+t_1}{2} = \Delta$, and the first moment of an even function about its line of symmetry equals the center times the mass, giving \eqref{eq:Picentroid}.
\end{proof}

\bibliographystyle{plain}
\bibliography{references}

@article{byrd1987,
  author  = {Byrd, R. H. and Nocedal, J. and Yuan, Y.},
  title   = {Global convergence of a class of quasi-{Newton} methods on convex problems},
  journal = {SIAM Journal on Numerical Analysis},
  volume  = {24},
  number  = {5},
  pages   = {1171--1190},
  year    = {1987},
  doi     = {10.1137/0724077}
}

@article{dai2002,
  author  = {Dai, Y.-H.},
  title   = {Convergence properties of the {BFGS} algorithm},
  journal = {SIAM Journal on Optimization},
  volume  = {13},
  number  = {3},
  pages   = {693--701},
  year    = {2002},
  doi     = {10.1137/S1052623401383455}
}

@article{dai2013,
  author  = {Dai, Y.-H.},
  title   = {A perfect example for the {BFGS} method},
  journal = {Mathematical Programming},
  volume  = {138},
  number  = {1--2},
  pages   = {501--530},
  year    = {2013},
  doi     = {10.1007/s10107-012-0522-2}
}

@article{liu1989,
  author  = {Liu, D. C. and Nocedal, J.},
  title   = {On the limited memory {BFGS} method for large scale optimization},
  journal = {Mathematical Programming},
  volume  = {45},
  pages   = {503--528},
  year    = {1989},
  doi     = {10.1007/BF01589116}
}

@article{mascarenhas2004,
  author  = {Mascarenhas, W. F.},
  title   = {The {BFGS} method with exact line searches fails for non-convex
             objective functions},
  journal = {Mathematical Programming},
  volume  = {99},
  number  = {1},
  pages   = {49--61},
  year    = {2004},
  doi     = {10.1007/s10107-003-0421-7}
}

@book{nocedalwright2006,
  author    = {Nocedal, J. and Wright, S. J.},
  title     = {Numerical Optimization},
  edition   = {2},
  publisher = {Springer},
  address   = {New York},
  year      = {2006},
  doi       = {10.1007/978-0-387-40065-5}
}

@inproceedings{powell1976,
  author    = {Powell, M. J. D.},
  title     = {Some global convergence properties of a variable metric algorithm
               for minimization without exact line searches},
  booktitle = {Nonlinear Programming},
  series    = {SIAM-AMS Proceedings},
  volume    = {9},
  editor    = {Cottle, R. W. and Lemke, C. E.},
  publisher = {SIAM},
  address   = {Philadelphia},
  pages     = {53--72},
  year      = {1976}
}

@article{powell2000,
  author  = {Powell, M. J. D.},
  title   = {On the convergence of the {DFP} algorithm for unconstrained
             optimization when there are only two variables},
  journal = {Mathematical Programming},
  volume  = {87},
  number  = {2},
  pages   = {281--301},
  year    = {2000},
  doi     = {10.1007/s101070050115}
}

@article{hestenes1952,
  author  = {Hestenes, M. R. and Stiefel, E.},
  title   = {Methods of conjugate gradients for solving linear systems},
  journal = {Journal of Research of the National Bureau of Standards},
  volume  = {49},
  number  = {6},
  pages   = {409--436},
  year    = {1952}
}

@article{broyden1970,
  author  = {Broyden, C. G.},
  title   = {The convergence of a class of double-rank minimization algorithms
             2. The new algorithm},
  journal = {IMA Journal of Applied Mathematics},
  volume  = {6},
  number  = {3},
  pages   = {222--231},
  year    = {1970},
  doi     = {10.1093/imamat/6.3.222}
}

@incollection{powell1984,
  author    = {Powell, M. J. D.},
  title     = {Nonconvex minimization calculations and the conjugate gradient method},
  booktitle = {Numerical Analysis ({Dundee}, 1983)},
  series    = {Lecture Notes in Mathematics},
  volume    = {1066},
  publisher = {Springer},
  address   = {Berlin},
  pages     = {122--141},
  year      = {1984},
  doi       = {10.1007/BFb0099521}
}

@article{wolfe1969,
  author  = {Wolfe, P.},
  title   = {Convergence conditions for ascent methods},
  journal = {SIAM Review},
  volume  = {11},
  number  = {2},
  pages   = {226--235},
  year    = {1969},
  doi     = {10.1137/1011036}
}

@incollection{zoutendijk1970,
  author    = {Zoutendijk, G.},
  title     = {Nonlinear programming, computational methods},
  booktitle = {Integer and Nonlinear Programming},
  editor    = {Abadie, J.},
  publisher = {North-Holland},
  address   = {Amsterdam},
  pages     = {37--86},
  year      = {1970}
}

@article{dixon1972,
  author  = {Dixon, L. C. W.},
  title   = {Quasi-{Newton} algorithms generate identical points},
  journal = {Mathematical Programming},
  volume  = {2},
  number  = {1},
  pages   = {383--387},
  year    = {1972},
  doi     = {10.1007/BF01584554}
}

@article{mascarenhas2014,
  author  = {Mascarenhas, W. F.},
  title   = {The divergence of the {BFGS} and {Gauss}--{Newton} methods},
  journal = {Mathematical Programming},
  volume  = {147},
  number  = {1},
  pages   = {253--276},
  year    = {2014},
  doi     = {10.1007/s10107-013-0720-6}
}

@article{li2001,
  author  = {Li, D.-H. and Fukushima, M.},
  title   = {A modified {BFGS} method and its global convergence in nonconvex minimization},
  journal = {Journal of Computational and Applied Mathematics},
  volume  = {129},
  number  = {1--2},
  pages   = {15--35},
  year    = {2001},
  doi     = {10.1016/S0377-0427(00)00539-7}
}

\end{document}